\documentclass[a4paper,11pt,oneside,reqno]{amsart}

\usepackage[utf8]{inputenc}
\usepackage{mathtools,amsmath,amssymb}
\usepackage{graphicx}
\usepackage{lmodern}
\usepackage[T1]{fontenc}
\usepackage[pagebackref=false]{hyperref}
\usepackage{bbm}
\usepackage{bm}
\usepackage{appendix}
\usepackage{comment}
\usepackage{mathtools}
\usepackage{array}
\usepackage{amsthm}
\usepackage{dsfont}
\usepackage{accents}
\usepackage[left=2.5cm,right=2.5cm,top=3cm,bottom=3cm]{geometry}
\usepackage{pgfplots}
\usepackage{tikz}
\usepackage{enumitem}
\usepackage{graphicx}
\usepackage{subcaption} 
\allowdisplaybreaks
\usepackage{tikz}
\usetikzlibrary{decorations.pathreplacing}

\newtheorem{theorem}{Theorem}[section]
\newtheorem{conjecture}[theorem]{Conjecture}
\newtheorem{lemma}[theorem]{Lemma}
\newtheorem{proposition}[theorem]{Proposition}

\newtheorem{assumption}{Assumptions}[theorem]

\theoremstyle{definition}
\newtheorem{remark}[theorem]{Remark}
\newtheorem{question}[theorem]{Question}
\newtheorem{definition}[theorem]{Definition}

\usepackage[backend=biber,style=alphabetic, maxnames=5, maxalphanames=5, sortcites=false, giveninits=true, isbn=false, clearlang=true, doi=false, url=false, sorting=nyt, labelnumber, defernumbers]{biblatex}
\renewbibmacro{in:}{}
\AtEveryBibitem{\clearlist{language}} 
\AtEveryBibitem{\clearfield{eprintclass}}
\AtEveryBibitem{\clearfield{note}}
\AtEveryBibitem{\clearfield{pages}}
\newcommand{\R}{\ensuremath{\mathbb{R}}}
\newcommand{\C}{\ensuremath{\mathbb{C}}}
\newcommand{\Z}{\ensuremath{\mathbb{Z}}}

\newcommand{\I}{\ensuremath{\mathbf{i}}}

\renewcommand{\rho}{\varrho}

\newcommand{\eps}{\varepsilon}
\renewcommand{\leq}{\leqslant}
\renewcommand{\geq}{\geqslant}

\numberwithin{equation}{section}

\begin{document}
	
	\title{Convergence of the KMP model to the KPZ equation}

	\author[G.~Barraquand]{Guillaume Barraquand}
	\address{G. Barraquand, Laboratoire de Physique de l'Ecole Normale Supérieure, Ecole Normale Supérieure, PSL University, CNRS, Sorbonne Université, Université Paris-Cité, 24 rue Lhomond, 75005 PARIS}
	\email{guillaume.barraquand@ens.fr} 
	
	\author[F.~Casini]{Francesco Casini}
	\address{F. Casini, Laboratoire de Physique de l'Ecole Normale Supérieure, Ecole Normale Supérieure, PSL University, CNRS, Sorbonne Université, Université Paris-Cité, 24 rue Lhomond, 75005 PARIS}
	\email{francesco.casini@ens.fr}

\begin{abstract}
We prove that the Kipnis-Marchioro-Presutti (KMP) process converges to the Kardar-Parisi-Zhang (KPZ) equation, as time $t$ goes to infinity, in a properly scaled observation window shifted by $t^{3/4}$. Our proof is based on identifying the KMP process with a stochastic flow of kernels  describing transition probabilities in a certain model of random walk in space-time random environment. This allows to apply a recent result of \cite{parekh2} proving convergence of the density field of random walks in random environment to the KPZ equation in a suitably general sense.
\end{abstract}

\maketitle

	\setcounter{tocdepth}{1}
	\tableofcontents

\section{Introduction}
Kipnis, Marchioro and Presutti introduced in \cite{kipnis1982heat} a stochastic model for energy exchange in a chain of harmonic oscillators, further studied in both mathematics and physics literature. The model can be described on an arbitrary graph as follows. Vertices are assigned random energies, evolving as a Markov process: for each pair of adjacent vertices, their energies are redistributed at exponential rate $1$, uniformly over all possible redistributions that conserve the total energy. 
\medskip 

This model, as well as variants and generalizations \cite{hiddenSymetries,modelsOfTransport,carinci2016asymmetric}, is often studied through the notion of Markov duality \cite{jansen2014notion,dualityBook}. The original article \cite{kipnis1982heat} actually proved one of the first examples of a so-called absorbing duality. Besides Markov duality, the model on a finite lattice connected to boundary reservoirs is one of the simplest  prototypes of driven diffusive system in non-equilibrium statistical physics. Recently, the model received increased  attention in the physics literature \cite{bettelheim2022full,bettelheim2022inverse,grabsch2022exact,rizkallah2023duality,krajenbrink2023crossover,bettelheim2024full}, in the context of macroscopic fluctuation theory (MFT). The latter  is  a general framework \cite{bertini2002macroscopic,bertini2015macroscopic} that extends the hydrodynamic theory to  predict the large deviation behaviour of many stochastic transport models. We shall also mention a new approach to characterizing the stationary measure of the model based on the so-called hidden temperature model \cite{de2024hidden},  new duality and intertwining relations \cite{giardina2025intertwining}, as well as a recent proof of the hydrodynamic limit  \cite{franceschini2025hydrodynamic}.  
\medskip 

The KMP model shares many similarities with the Symmetric Simple Exclusion Process (SSEP). In particular, the MFT equations -- a couple of nonlinear PDEs from which one can conjecturally express large deviation rate functions -- are exactly the same for both models \cite{zarfaty2016statistics, grabsch2022exact,rizkallah2023duality}. This similarity does not concern only large deviations. On the graph $\Z$, both the SSEP starting from a step initial condition and the KMP model starting from a localized energy distribution are expected to converge to the Edwards-Wilkinson equation (additive noise stochastic heat equation). Indeed, this is proven for the SSEP  \cite{landim-book} 
and several variants of KMP model \cite{10.1214/EJP.v3-28,balazs2006random,yu2016edwards}.  

\medskip 
The goal of this article is to show that despite the similarities between SSEP and KMP models, the scaling limit of the KMP model is richer in the sense that in the appropriate scaling window, its fluctuations converge to the multiplicative noise stochastic heat equation  \eqref{SHE-general}, that is the exponential of the Kardar-Parisi-Zhang \eqref{KPZ} equation  (Theorem \ref{Thm-convergence}). The main idea behind  our proof is that the field of energies in the KMP model should be understood as a stochastic flow that describes the probability distribution of a random walk in random environment (Proposition \ref{Proposition-KMP-SFK}). {\color{black}
To our knowledge, this identification of an energy exchange model with a random walk in random environment is new. A discrete-time analogue for a brick-wall variant of the KMP model, also new, is discussed in Appendix \ref{appendix-DtModels}, where we explain that it may be viewed as a marginal of the directed wave model of \cite{kardar1992}, which was an important source of inspiration for this work.}
Once this identification is made, we can apply a recent result of \cite{parekh2} proving convergence to (mSHE) for a general class of random walks in (dynamic) random environment. A special case of this result is stated below as Theorem \ref{Thm-parekh}.

\medskip 
We expect that up to model-dependent constants, our main result also holds for more general redistribution mechanisms involving a greater and possibly random number of consecutive sites,  with a similar proof. Further, our work also suggests Conjecture \ref{conjecture}: for any velocity $v\in (-1,1)$, $v\neq 0$,  as time $t$ goes to infinity,  the logarithm of KMP energies around site $\lfloor vt\rfloor$ converges to the Tracy Widom distribution.

\medskip 
 Before commenting further on the connections between the KPZ equation and random walks in random environment, we should mention that a connection between the KMP model and the KPZ equation is already studied in \cite{krajenbrink2023crossover}. Indeed, based on MFT predictions, the large deviations in the SSEP and KMP model, or any  other model with the same mobility and  diffusion coefficient,  should be related to a stochastic PDE describing the probability distribution of a Brownian particle in a white noise velocity field. The latter is known to be related to the Kardar Parisi-Zhang universality class \cite{leDoussal-17, barraquand2020large, B-LD-moderate}. More precisely,  \cite{krajenbrink2023crossover} studied the crossover between the large deviation of the KMP model and the small time large deviation behaviour of the KPZ equation. 
 
\subsection*{KPZ universality class and random walks in random environment} 

 The KPZ equation,  introduced in \cite{KPZ-original}, is an emblematic model in the KPZ class. Although it  is distinct from the universal KPZ fixed point \cite{matetski2021kpz}, it arises in general as scaling limit of models when one can rescale simultaneously time and space as well as a parameter of the system that controls the asymmetry of the model, or the strength of the noise. For example, the weakly asymmetric simple exclusion process converges to the KPZ equation \cite{bertini1997stochastic}. When tuning the strength of the noise, that is scaling the temperature to infinity, the free energy of directed polymers converges to the KPZ equation \cite{alberts2014intermediate}.

\medskip 
{\color{black}Random walks in a lattice in dynamical i.i.d. environment were studied in the series of papers \cite{BoldrighiniIgnatyukMalyshevPellegrinotti1992, BoldrighiniMinlosPellegrinotti1997,BernabeiBoldrighiniMinlosPellegrinotti1998, boldrighini1999central, boldrighini2001t, bernabei2001anomalous, boldrighini2004random}. Under some assumptions on the environment disorder, these references proved the almost-surely diffusive behaviour of RWRE, and studied  the scale of random corrections due to disorder. Assumptions were relaxed in a later series of works \cite{berard2004almost, rassoulagha2005almost, balazs2006random, rassoulagha2009almost} and large deviations principles were considered \cite{rassoulagha2013quenched}, in parallel to similar results about the limiting free energy of directed polymer models.}

 \medskip 
 A toy model of discrete time random walk in space-time i.i.d. random environment was introduced in \cite{barraquand2015}. When transition probabilities are beta distributed, the model is exactly solvable, which allows to prove that the quenched probabilities of  deviations of order $t$ for the random walk have Tracy-Widom distributed fluctuations on the scale $t^{1/3}$.  This is a signature of the KPZ universality class. A number of further results on the model were obtained in subsequent works \cite{thiery2015integrable, thiery2016exact, corwin2017kardar,balazs2018large,  barraquand2020large, barraquand2022, hass2023anomalous, hass2024extreme,hass2024first}.
  In the context of the present paper, the most important development is the fact that the quenched probabilities of moderate deviations of order $t^{3/4}$ have fluctuations which are described by the KPZ equation. This was predicted, with increasing degree of precision, in \cite{thiery2016exact, leDoussal-17, B-LD-moderate} and then proved in \cite{das2024multiplicative} for rather general discrete time simple random walks (see also \cite{das2024kpz} and  \cite{brockington2022edge} for similar results). 
 
 \medskip 
 It is natural to wonder about the universality of this convergence, whether it remains true for more general non simple random walks in random environment, possibly allowing some correlations in space that vanish at large scale. In {\color{black}physics} references, universality was studied in \cite{hass2025universal} (see also \cite{hass2024extreme, hass2024first, das2024multiplicative}). Interestingly, the scalings depend on the model: it was found in \cite{hass2025super} that  if the average $\mathrm E[X(t+1)-X(t)]$ of an  elementary step of the random walk $X(t)$ is deterministic (with respect to the law of the environment), that is if the randomness of the environment is felt only in the higher cumulants of a random walk elementary step, the scaling of moderate deviations should be different. This prediction was confirmed in \cite{parekh2}, showing that for a general class of $\R$-valued random walks in random environment, moderate deviations converge to the KPZ equation on a scale that depends on some integer $p$. The correct scale of moderate deviations to consider is $n^{\frac{4p-1}{4p}}$ where $p$ is the smallest integer so that the $p$-th moment $\mathrm E[(X(t+1)-X(t))^p]$ of an elementary step is random with respect to the law of the environment. The general convergence result of \cite{parekh2} is stated below as Theorem \ref{Thm-parekh} in the case $p=1$, which is the only case we need. In the case $p=2$, the coefficient $7/8$ can also be predicted by considering a continuous model of diffusion with space-time white noise diffusion coefficient \cite{ledoussal2023private}.

 \subsection*{Outline}
 In Section \ref{Section-KMP-SF-general} we introduce a random walk in random environment (RWRE) and we show that the energy variable of the KMP process can be described by the quenched transition kernel of this RWRE. Moreover, we show that the $k$-point motion is described by the so-called dual-KMP process and recover the Markov duality relation \cite{kipnis1982heat,modelsOfTransport} from the reversibility of the model. In Section \ref{Section-mainResult} we state our main result (Theorem \ref{Thm-convergence}), which becomes a direct consequence of a general convergence result proven in \cite{parekh2} which is recorded below as Theorem \ref{Thm-parekh}. We then formulate some possible generalizations and conjectures suggested by our results. The following sections are devoted to the proof of Theorem \ref{Thm-convergence}. In Section \ref{section-assumptions} we show that the kernel of our RWRE {\color{black}satisfies} the assumptions required in \cite{parekh2}. In Section \ref{section-skorokhod}, we show how to prove the convergence of the field in $D([0,T],\mathcal{S}'(\R))$, starting {\color{black}from} the convergence for the interpolated field. In Section \ref{Section-noiceVariance} we construct a time discretization of our RWRE that allows to compute explicitly the value of the variance of the noise in the limit.
 
 We also discuss discrete analogues in Appendix \ref{appendix-DtModels}. We recall the definition of the discrete time Beta-RWRE introduced in \cite{barraquand2015}, explain how it is related to a discrete variant of the KMP model with a brick-wall structure, and we explain connections of both models to a model for the unitary evolution of a quantum wavefunction, introduced in \cite{kardar1992} as a model of directed waves.

\subsection*{Acknowledgments} G.B. and F.C. were partially supported by Agence Nationale de la Recherche through the grant ANR-23-ERCB-0007. G. B. was also  supported by ANR grant ANR-21-CE40-0019. G.B. is grateful to Pierre Le Doussal for drawing his attention to  the reference \cite{kardar1992} and for many useful conversations related to the models discussed in Appendix \ref{appendix-DtModels}. G.B. also thanks Ivan Corwin for suggestions related to Section \ref{sec:furtherdirections}.

\section{KMP model as a stochastic flow}\label{Section-KMP-SF-general}
\subsection{The KMP model}
We recall the definition of the KMP model on $\mathbb{Z}$. 
\begin{definition}[\cite{kipnis1982heat}]\label{definition-KMP}
    The KMP model $\left(\eta(t,x)_{x\in \mathbb{Z}}\right)_{t \geq 0}$ is a continuous-time Markov process defined on the state space $\Omega=\mathbb{R}_{+}^{\mathbb{Z}}$. The infinitesimal generator is given by 
    \begin{equation*}
        \mathcal{L}=\sum_{x\in \mathbb{Z}}\mathcal{L}_{x,x+1}\,,
    \end{equation*}
    where the action of the linear operators $\mathcal{L}_{x,x+1}$ on local functions $f:\Omega\to \mathbb{R}$ reads
    \begin{equation*}
    \begin{split}
 \left(\mathcal{L}_{x,x+1}f\right)(\eta)=\int_{0}^{1}&\left[f(\ldots,\eta(x-1),\,p(\eta(x)+\eta(x+1)),\,(1-p)(\eta(x)+\eta(x+1)),\,\eta(x+2),\ldots)\right.
 \\&\left.\quad-f(\eta)\right]\frac{p^{\alpha-1}(1-p)^{\alpha-1}}{\text{B}(\alpha,\alpha)}dp \,.
  \end{split}
 \end{equation*}
\end{definition}
In words, at exponential rate $1$ on each bond $(x,x+1)$ the total energy $\eta(x)+\eta(x+1)$  is redistributed according to a $\text{Beta}(\alpha,\alpha)$-distributed random variable, i.e. the couple $(\eta(x), \eta(x+1))$ is replaced by 
$$ (B (\eta(x)+ \eta(x+1)), (1-B)(\eta(x)+ \eta(x+1)) )$$
where $B\sim \text{Beta}(\alpha,\alpha)$. More precisely, the process may be constructed as follows. Associate to each bond $(x,x+1)$ a Poisson process  $(T_{x,x+1}^{(j)})_{j\in \{0,1,\ldots\}}\in \mathbb{R}^{+}$. 
For each Poisson time $T= T_{x,x+1}^{(j)}$, let  $B_{x,x+1}^{(j)}$ be an independent random variable distributed as  Beta$(\alpha,\alpha)$. On each bond, at each Poisson times $T$, the energies are updated  by letting
\begin{equation}
	(\eta(T,x), \eta(T,x+1)) = \left( B  (\eta(T^-, x)+ \eta(T^-, x+1)), (1-B)(\eta(T^-, x)+ \eta(T^-, x+1)) \right) 
	\label{updating-rule-KMP}
\end{equation} 
where $\eta(t^-,x)$ denotes the left-limit at time $t$, i.e. the value just before the jump. The process may be constructed using a variant of Harris graphical construction \cite{10.1214/EJP.v3-28}.
\begin{remark}  In the original paper \cite{kipnis1982heat}, the KMP model was introduced with a $\text{Uniform}([0,1])$-random variable (corresponding to $\alpha=1$). It has been generalized in \cite{modelsOfTransport} to $B\sim \text{Beta}(\alpha,\alpha)$, which is sometimes called generalized KMP process.
\end{remark} 
\begin{remark} 
The KMP model can be interpreted as a system of particles with positions $(\mathcal R_n(t))$, see \cite{10.1214/EJP.v3-28} and \cite{grabsch2022exact}. Consider the special case $\alpha=1$ for simplicity. From a KMP energy configuration, we can define a particle configuration (up to a global shift) via 
\begin{equation*}
	\eta(n,t):=\mathcal{R}_{n+1}(t)-\mathcal{R}_{n}(t)\,.
\end{equation*}
The particle's dynamic is such that each particle carries a Poisson clock, and when it rings, $\mathcal R_n(t)$ jumps to a location chosen uniformly in the interval $[\mathcal{R}_{n-1}(t),\mathcal{R}_{n+1}(t)]$. 
\end{remark}

\subsection{A continuous-time random walk in  random environment}\label{subsection-RWRE}
We now define a specific model of random walk in random environment (RWRE). The  position of the random walk at time $t\in \mathbb{R}_{+}$ is denoted $(X(t))_{t\geq 0}$. The environment is made of a family of Poisson processes (indexed by bonds of $\mathbb Z$) and Beta variables for each Poisson time on each bond. 
\begin{definition} \label{definition-RWRE}
	On a probability space $(\Omega, \mathcal F, \mathbb P)$, we define the following random variables: for each bond $(x,x+1)$, let  $\left(T_{x,x+1}^{(j)}\right)_{j\in \mathbb{Z}_{\geq 0}}$ be a Poisson point process,   independent for each bond. 
	
For each bond $(x,x+1)$ and each $j\in \mathbb N$,  let  $B_{x,x+1}^{(j)}$ be an independent random variable distributed as  Beta$(\alpha,\alpha)$ as previously. 
We denote a realization of the  environment by $(\mathcal P, \omega)\in \Omega$, where  $\mathcal{P}$ is a realization of the Poisson processes and $\omega$ is a realization of the Beta variables.
	
	Given $(\mathcal P, \omega)\in \Omega$, the random walk dynamic is the following. If at time $t$, $X(t)=x$, we have to consider the first  Poisson times larger than $t$ associated to the bonds $(x-1,x)$ and $(x,x+1)$.  Let
	\begin{equation*}
		T_{1}:=\min_{j\in \mathbb{Z}_{\geq 0}}\left\{T_{x-1,x}^{(j)}\,:\, T_{x-1,x}^{(j)}>t\right\},\qquad T_{2}:=\min_{j'\in \mathbb{Z}_{\geq 0}}\left\{T_{x,x+1}^{(j')}\,:\, T_{x,x+1}^{(j')}>t\right\}\,.
	\end{equation*}
	Then, we have that:
	\begin{itemize}[leftmargin=12pt]
		\item if $T_{1}>T_{2}$, the walker remains at site $x$ with probability $B_{2}$ or goes to $x+1$ with probability $\left(1-B_{2}\right)$, where $B_2$ is the Beta variable associated with the Poisson time $T_2$;
		\item if $T_{1}<T_{2}$, the walker goes to $x-1$ with probability $B_{1}$ or remains at site $x$ with probability $\left(1-B_{1}\right)$,  where $B_1$ is the Beta variable associated with the Poisson time $T_1$.
	\end{itemize}
\end{definition}
In words, we may say that when the Poisson clock rings on a fixed bond, the walker goes to the left-most site of the bond with probability $B_{T_{\cdot}}$, or it goes to the right-most one with probability $1-B_{T_{\cdot}}$. 
Figure \ref{fig:Harris} shows a possible path taken by the random walk.
\begin{figure}[ht]
    \centering
    \includegraphics[width=0.85\linewidth]{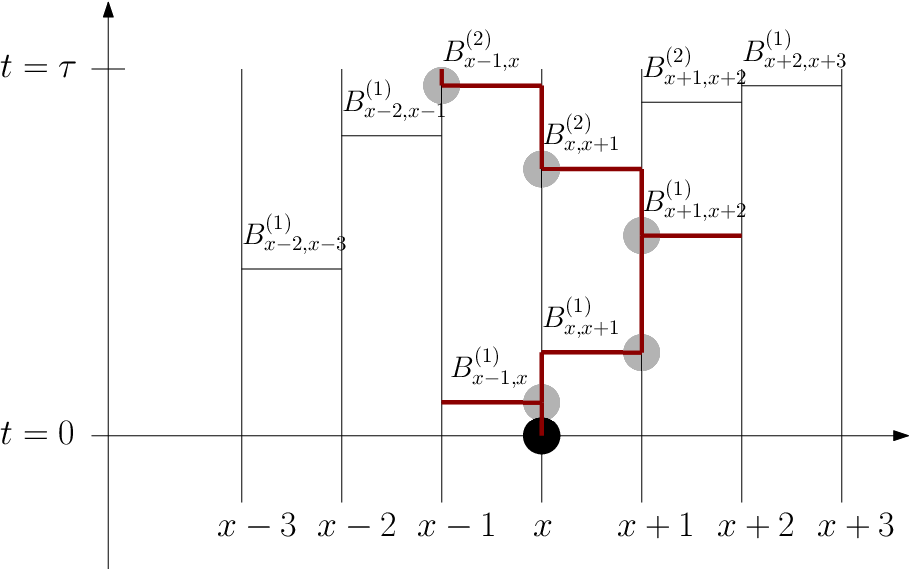}
    \caption{
    	The horizontal segments (bridges) represent the times of the  Poisson point process attached to each bond. They are decorated with random variables $B\sim \text{Beta}(\alpha,\alpha)$. The continuous red line is a possible trajectory of random walk, from $(0,x)$ to $(\tau,x-1)$. The gray circles represent the location of the particle after each Poisson event.}
    \label{fig:Harris}
\end{figure}

The law of the RWRE can be described by the quenched transition kernel 
\begin{equation}\label{definition-kernel}
    K_{s,t}(y,x):=\mathrm{P}^{\mathcal{P}, \omega}\left(X(t)=x|X(s)=y\right)\,.
\end{equation}
where $\mathrm{P}^{\mathcal{P},\omega}(\cdot)$ is the  probability measure of the random walk conditioning on $(\mathcal{P}, \omega)$. We sometimes use the shorthand notation  $K_{t}(x,y):=K_{0,t}(x,y)$. 

The dynamic of the RWRE $(X(t))_{t\geq 0}$, yield dynamic for its transition kernel. For some $t\geq 0$, we consider the transition kernel $(K_{t}(y,x))_{x\in\mathbb{Z}}$, for a particle initially located at site $y$. In the same set up as before, the updating rule consists of:
\begin{itemize}
    \item if $T_{1}>T_{2}$
    \begin{equation}\label{updating-rule-kernel1}
        K_{T_{2}}(y,x)=B_{2}\left(K_{t}(y,x)+K_{t}(y,x+1)\right)\;;
    \end{equation}
    \item if $T_{1}<T_{2}$
    \begin{equation}\label{updating-rule-kernel2}
        K_{T_{1}}(y,x)=(1-B_{1})\left(K_{t}(y,x-1)+K_{t}(y,x)\right)\,.
    \end{equation}
\end{itemize}
The transition kernel $K_{s,t}(y,x)$ defined in \eqref{definition-kernel} is a stochastic flow of kernels, 
in the sense of Definition 5.1 of \cite{brownianWeb}. This means that it satisfies the conditions:
\begin{enumerate}
	\item Consider $x,y\in \mathbb{Z}$ and $s\leq t\leq u$, then a.s. we have that $K_{s,s}(x,y)=\delta_{x}(y)$ and 
	\begin{equation*}
		K_{s,u}(y,x)=\sum_{z\in\mathbb{Z}}K_{s,t}(y,z)K_{t,u}(z,x)\,.
	\end{equation*}
	\item Consider $t_{0}<t_{1}<\ldots t_{n-1}<t_{n}$, then $\left(K_{t_{i-1},t_{i}}(\cdot,\cdot)\right)_{i\in\{1,\ldots,n\}}$ are all independent. 

	\item The kernels $K_{s,t}(\cdot,\cdot)$ and $K_{s+u,t+u}(\cdot,\cdot)$ are equal in finite dimensional distribution for each $s\leq t$ and $u$. 

\end{enumerate}

\subsection{The KMP as a stochastic flow of kernels}\label{section-KMP-SFK}
As is probably already clear from \eqref{updating-rule-kernel1} and \eqref{updating-rule-kernel2}, the transition kernel of the RWRE defined above evolves as the KMP energies. 
\begin{proposition}\label{Proposition-KMP-SFK}
	Let $\left(\eta(t,x)_{x\in \mathbb{Z}}\right)_{t\geq 0}$
be the KMP process with initial condition $\left(\eta(0,x)\right)_{x\in \mathbb{Z}}$.
	Let $\left(K_{t}(y,x)_{x,y\in \mathbb{Z}}\right)_{t\geq 0}$ be the stochastic flow of kernels defined in \eqref{definition-kernel}. For any bounded function $\rho_0:\mathbb{Z}\to \mathbb{R}_{+}$, let 
    \begin{equation}\label{rho-t}
 \rho_{t}(x):=\sum_{y\in \mathbb{Z}}\rho_{0}(y)K_{t}(y,x).
    \end{equation}
	Then, if $\rho_{0}=\eta(0, \cdot)$ we have that 
	\begin{equation*}
		\eta(t,x)\overset{(d)}{=}\rho_{t}(x), 
	\end{equation*}
	 in the Skorokhod space $D\left(\mathbb{R}_{+},\mathbb{R}_{+}^{\mathbb{Z}}\right)$.
\end{proposition}
\begin{proof}
	Let us first consider the initial condition $\eta(0,x)=\mathbbm{1}_{0}(x)$, we have 
	\begin{equation*}
		\eta(0,x)=K_{0}(0,x)=\mathbbm{1}_{0}(x)\,.
	\end{equation*}
	Comparing the dynamic of the KMP model \eqref{updating-rule-KMP} and those of  the transition kernel  \eqref{updating-rule-kernel1}-\eqref{updating-rule-kernel2}, and assuming that the two stochastic processes share the same Poisson events and the same Beta variables, we have $\eta(t,x)= K_{t}(0,x)$ for all $t\geq 0$ and $x\in \mathbb Z$. 
	For an arbitrary initial condition $\rho_{0}=\eta(\cdot,0)$, defining $\rho_{t}(x)$ as \eqref{rho-t}, the Proposition follows by linearity. 
\end{proof}
As stationary measures for the KMP model are well-known \cite{kipnis1982heat,modelsOfTransport}, we obtain the following. 
\begin{lemma}\label{Lemma-invariantMeasure-K}
	Let $\mu_{\alpha}$ be the product measure on configurations  $(\eta(x))_{x\in \Z} \in \R_+^{\Z}$ such that the $\eta(x)$ are i.i.d. and  $\mathrm{Gamma}(\alpha)$ distributed. Then, the Markov process $\rho_t$ defined in \eqref{rho-t} has a unique stationary measure (up to multiplicative constant) given by $\mu_{\alpha}$. 
\end{lemma}
\begin{proof}
Given Proposition \ref{Proposition-KMP-SFK}, this is an immediate consequence of the known stationary measure for the KMP model. 
\end{proof}
\subsection{$k$-point motion} 

The description of the KMP model as a stochastic flow of transition kernels provides a new perspective on the Markov duality between the KMP model and the following model introduced in  \cite{kipnis1982heat,modelsOfTransport}. 
\begin{definition}[Dual-KMP]
	\label{def:dualKMP}
    The Dual-KMP model $(\xi(t,x)_{x\in\mathbb{Z}})_{t \geq 0}$ is a continuous-time Markov process defined on the state space $\mathbb{Z}_{\geq 0}^{\mathbb{Z}}$, where $\xi(t,x)$ represents a number of particles on site $x$ at time $t$. At rate $1$, on each bond $(x,x+1)$, the total  number of particles $\xi(x)+\xi(x+1)$ is redistributed according to a random variable  $R \sim  \mathrm{BetaBinomial}(n, \alpha, \alpha)$ random variable with  $n=\xi(x)+\xi(x+1)$, such that at every Poisson time $T$,  $(\xi(T^-,x), \xi(T^-,x+1))$ becomes
    $$ (R, \xi(T^-,x)+ \xi(T^-,x+1)-R)\,. $$
    We recall that a random variable $R\sim \text{BetaBinomial}(n,\alpha,\beta)$, if it has density given by $$ \mathbb{P}(R=k)=\binom{n}{k}\int_{0}^{1}p^{\alpha+k-1}(1-p)^{\beta+n-k-1}dp\,.$$ 
\end{definition}

The KMP and the dual-KMP satisfy a Markov duality, with respect to the duality function (see \cite{kipnis1982heat,modelsOfTransport})
\begin{equation*}
    D(\eta,\xi):=\prod_{x\in \mathbb{Z}}d(\eta(x),\xi(x)), \hspace{10pt}
    d(\eta(x),\xi(x))=\eta(x)^{\xi(x)}\frac{\Gamma(\alpha)}{\Gamma(\alpha+\xi(x))}\,.
\end{equation*}
Namely, the following equation holds true:
\begin{equation}\label{duality-relation}
    \mathbb E_{\eta(0,\cdot)}\left[ D(\eta(t,\cdot),\xi(0,\cdot))\right]=\mathbb E_{\xi(0,\cdot)}\left[ D(\eta(0,\cdot),\xi(t,\cdot))\right]\,,
\end{equation}
where the expectations on the right-hand-side and on the left-hand-side are taken with respect to the law of the KMP with initial condition $\eta(0,\cdot)$ and of the dual KMP with initial condition $\xi(0,\cdot)$, respectively. An alternative proof of Markov duality, based on reversibility, is reported in Appendix \ref{sec:proofduality}.

We observe that the redistribution of the dual particles on each bond can be seen as binomial experiments with random probability of success with law $\text{Beta}(\alpha,\alpha)$. Therefore, when the Poisson events occurs, each of the particles in the bond behaves as an independent random walk in a Beta random environment. 

Let $k\in \mathbb{Z}_{\geq 0}$ and call $k$-point motion the joint evolution of $\left\{(X_{1}(t),\ldots,X_{k}(t))\right\}_{t\geq 0}$, where for all $i\in \{1,\ldots,k\}$ $X_{i}(t)$ is a copy of the RWRE initially located on an arbitrary site of $\mathbb{Z}$.
Therefore, we observe that the dual-KMP process can be interpreted as the $k$-point motion associated with the RWRE, in which the kernel of each of these walkers evolves as the KMP process (properly initialized). 

The transition kernel for the $k$-points  motion to go from $(x_{1},\ldots,x_{k})\in \mathbb{Z}^{k}$ to $(y_{1},\ldots,y_{k})\in \mathbb{Z}^{k}$ is defined as
\begin{equation}\label{k-pts-kernel}
	\mathbf{p}_{t}^{(k)}((x_{1},\ldots,x_{k})\to (y_{1},\ldots,y_{k})):=\mathbb{E}\left[K_{t}(x_{1},y_{1})\cdots K_{t}(x_{k},y_{k})\right]\,.
\end{equation} 

\begin{lemma}\label{Lemma-invp2}
	Let $\nu^{\text{inv}}_{\alpha}$ be a measure over $\mathbb{Z}^{k}$, with densities $$\nu^{\text{inv}}_{\alpha}(x_{1},\ldots,x_{k})=\mathbb{E}\left[\prod_{i=1}^{k}\mu_{\alpha}(x_{i})\right],$$
	where, with a slight abuse of notation, we have used again the notation $\mathbb E$ to denote the expectation with respect to the Gamma variables in the definition of the stationary measure $\mu_{\alpha}$.  Then, the $k$-points motion $(X_{1}(t),\ldots,X_{k}(t))_{t\geq 0}$ has a unique stationary measure given by $\nu^{\text{inv}}_{\alpha}$. 
\end{lemma}
\begin{proof}
	The proof is an immediate consequence of Lemma \ref{Lemma-invariantMeasure-K}. 
\end{proof}
The duality relation \eqref{duality-relation} can now be re-obtained starting from reversibility and using the kernel \eqref{k-pts-kernel} for the $k$-points motion. Even though the result is well-known, we present a proof using the $k$-point motion in Appendix \ref{sec:proofduality}.  
\section{Main result}

\subsection{Kardar-Parisi-Zhang equation}\label{Section-mainResult}
The Kardar-Parisi-Zhang (KPZ) is a stochastic partial differential equation (SPDE), introduced in \cite{KPZ-original}, which governs the evolution of a function $h(t,x)$, where $t\in \mathbb{R}_{+}$ and $x\in \mathbb{R}$. This SPDE reads
\begin{equation}
    \partial_{t} h(t,x)=\frac{1}{2}\partial_{xx}h(t,x)+\frac{1}{2}\left(\partial_{x}h(t,x)\right)^{2}+\Xi(t,x)\,, 
    \label{KPZ}
    \tag{KPZ}
\end{equation}
where $\Xi(t,x)$ is a space-time white noise. We consider solutions to the KPZ equation \eqref{KPZ} in the Cole-Hopf sense. For any fixed time horizon $T>0$ and initial condition $h_0\in C(\R,\R)$, a random function $h\in C([0,T], C(\R,\R))$ is a solution to \eqref{KPZ} if  $h(t,x):=\log(\mathcal{U}(t,x))$, where $\mathcal{U}(t,x)$ solves the multiplicative-noise stochastic heat equation
\begin{equation}\label{SHE-general}\tag{mSHE}
    \partial_{t}\mathcal{U}(t,x)=\frac{1}{2}\partial_{xx}\mathcal{U}(t,x)+\mathcal{U}(t,x)\Xi(t,x)\,,
\end{equation}
with initial condition $\mathcal{U}_{0}(x)=e^{h_{0}(x)}$. A function $\mathcal{U}(t,x)$ is a {\color{black}mild} solution to \eqref{SHE-general} with initial condition $\mathcal{U}_{0}(x)$ if it satisfies 
\begin{equation*}
    \mathcal{U}(t,x)=\int_{\mathbb{R}}p_{t}(x-y)\mathcal{U}_{0}(y)dy+\int_{0}^{t}ds\int_{\mathbb{R}}p_{t}(x-y)\mathcal U(s,y)\Xi(s,y)\,.
\end{equation*}
where $p_{t}(x)$ denotes the heat kernel. We will consider a {\color{black} different initial condition} where $h_0$ is not a continuous function but rather the so-called narrow wedge solution, that is when $h(t,x) = \log \mathcal U(t,x)$ and  $\mathcal{U}_{0}(x)=\delta_{0}(x)$. There exists a unique solution to \eqref{SHE-general} (see \cite{bertini1995stochastic, parekh2019-BD}) in the class of processes in $C((0,T], C(\R,\R))$ adapted to the filtration generated by the white noise $\Xi$, and satisfying the $L^2$ bound 
\begin{equation*}
    \sup_{x\in \mathbb{R},\, t\in(0,T]}t\mathbb{E}\left[\mathcal{U}(t,x)^{2}\right]<\infty\,.
\end{equation*}

\subsection{Main result}\label{section-MainResult}
Consider the KMP process $(\eta(t,x)_{x\in \mathbb{Z}})_{t\geq 0}$, with initial condition $\eta(0,x)=\mathbbm{1}_{\{x=0\}}$. For $N\in \mathbb{Z}_{\geq 0}$, for $t\in [0,T]$ 
and for $x\in N^{-1/2}\mathbb{Z}-N^{1/4}t$ we introduce the density field 
\begin{equation}\label{distribution-field}
\mathfrak{F}_{N}\left(t,x\right)=C_{N,t,x}\; \eta(t N, N^{3/4}t+N^{1/2}x)\,,
\end{equation}
where 
\begin{equation}\label{filed-explicit}
	C_{N,t,x}=e^{N^{1/4}x + N^{1/2}\frac{t}{2}{\color{black}-\frac{t}{24}}}\,.
\end{equation}
This density field should be seen as a distribution which can be  integrated against test functions $\phi\in C_{c}^{\infty}(\mathbb{R})$ as 
\begin{equation}\label{skorokhod-field}
	\left\langle \mathfrak{F}_{N}\left(t,\cdot\right),\phi\right\rangle =\sum_{x\in\mathbb{Z}-N^{3/4}t}\mathfrak{F}_{N}\left(t,\frac{x}{\sqrt N}\right)\phi\left(\frac{x}{\sqrt N}\right)\,, 
\end{equation}
where the shift by $N^{3/4}t$ is simply due to the fact that $\eta(t, \cdot)$ is defined on integers. 
We recall the definition of the Schwartz space 
\begin{equation*}
	\mathcal{S}(\mathbb{R}):=\left\{f\in C_{c}^{\infty}(\mathbb{R})\,:\,\forall \alpha,\beta\in \mathbb{Z}_{\geq 0},\; \sup_{x\in \mathbb{R}}\lvert x^{\alpha} (D^{\beta}f)(x)\rvert<\infty\right\}\,, 
\end{equation*}
and denote its dual space by $\mathcal{S}^{'}(\mathbb{R})$.
 \begin{theorem}\label{Thm-convergence}
The sequence  $(\mathfrak{F}_{N})_{N\in\mathbb{Z}_{\geq 0}}$ is tight with respect to the topology $D\left([0,T],\mathcal{S}^{'}(\mathbb{R})\right)$. Furthermore, any limit point lies in $C\left((0,T],C(\mathbb{R})\right)$ and coincides with the law of the It\=o solution to the multiplicative noise stochastic heat equation given by  
\begin{equation}\label{SHE}
	\begin{cases}
		\partial_{t}\mathcal{U}(t,x)=\frac{1}{2}\partial_{xx}\mathcal{U}(t,x)+\frac{1}{\sqrt{{\color{black}2}\alpha}}\mathcal{U}(t,x)\Xi(t,x)\\
		\mathcal{U}(0,x)=\delta_{0}(x)
	\end{cases}\,.
\end{equation}
Here, $\alpha>0$ is the parameter of the $\text{Beta}(\alpha,\alpha)$ distribution of the random environment and $\Xi(t,x)$ is the standard Gaussian space-time white noise.  
\end{theorem}
{\color{black}The proof of Theorem \ref{Thm-convergence} is based on the results of \cite{parekh2}. This paper generalizes the results of \cite{das2024multiplicative} to discrete time  stochastic flows of kernels under generic assumptions, thus applying to a wide family of RWRE transition probabilities. The distribution-valued field constructed from the discrete-time quenched transition probabilities as in \eqref{skorokhod-field} is proven to converge weakly to the mSHE \eqref{SHE-general} with $\delta_{0}$ initial condition (see the precise statement in Theorem \ref{Thm-parekh} below). In \cite{parekh2},  the convergence holds whenever the family of stochastic flow of kernels satisfies six assumptions:  temporal independence, spatial translation invariance of the random kernels, some mild bound on the annealed random walk law, non-degeneracy of the environment,  sufficiently fast decorrelation
of the $k$-point motion, and irreducibility of the $2$-point motion.}

\begin{proof} 
Since the KMP process can be identified with the stochastic flow of kernels \eqref{definition-kernel}, by   Proposition \ref{Proposition-KMP-SFK}, we may apply the result of \cite{parekh2} (reported it in the case $p=1$ in Theorem \ref{Thm-parekh}), specialized to our particular family of stochastic flows. 
More {\color{black}precisely}, in Section \ref{section-assumptions} we verify that the assumptions of \cite{parekh2} are satisfied by our transition kernel (Proposition \ref{proposition-assumptions-verify}). For clarity and completeness,  we first  recall and list these hypothesis in Assumptions \ref{assumption} in Section \ref{section-generalConvergence}. This yields the convergence of the field obtained by linearly interpolating the values at discrete times. In Proposition \ref{proposition-skorokhod}, we prove the convergence in $D\left([0,T],\mathcal{S}^{'}(\mathbb{R})\right)$. Finally, in Section \ref{Section-noiceVariance}, we prove that $\gamma = \frac{1}{\sqrt{{\color{black}2}\alpha}}$ (Proposition \ref{proposition-limitingConstant}).  
\end{proof}
\subsection{Further directions} 
\label{sec:furtherdirections} 
 Based on the convergence of the Beta RWRE to the Tracy-Widom distribution, and given Proposition \ref{Proposition-KMP-SFK}, we make the following conjecture.
\begin{conjecture}
	\label{conjecture}
	Let $v\in (-1,1)$, $v\neq 0$. Let $(\eta(t,x))_{x\in Z}$ be the energies in  the KMP model with initial condition $\eta(0,x)=\mathds{1}_{x=0}$. There exists deterministic functions $f(v)$ and $\sigma(v)$ such that 
	$$ \lim_{t\to\infty} \mathbb P\left(\frac{\log \eta(t,vt)+f(v)t}{\sigma(v) t^{1/3}}\leq s\right)=F_{TW}(s)  $$
	 where $F_{TW}$ is the cumulative distribution of the Tracy-Widom GUE distribution \cite{tracy1992level}. 
\end{conjecture}
More generally, we expect that the field 
$$ H_n(t,x)  := \frac{\log \eta(tn,vtn + n^{2/3}x) + f(t,n,v)}{g(t,n,v)} $$
should converge as $n \to\infty$ to the KPZ fixed point (introduced in \cite{matetski2021kpz}) for well-chosen functions $f$ and $g$. The general picture is the following. We refer to Figure \ref{fig:noisyheat} and to  \cite{B-LD-moderate} for similar explanations in the RWRE setting. When $x$ is of order $t^{1/2}$,  $\eta(t,x)$ behaves on average as $e^{-x^2/(2t)}$  with site to site fluctuations given by the stationary measure, and lower order fluctuations conjecturally given by the additive noise stochastic heat equation.  For $\vert v\vert >1$, it is easy to show that $\eta(t,vt)$ is equal to zero with very high probability: indeed, starting from $\eta(0,x)=\mathds{1}_{x=0}$, the energy at site $x$ becomes nonzero only after a random time distributed as the sum of $x$ independent exponential random variables. For $0<\vert v\vert <1$, $\eta(t,vt)$  is nonzero, and it asymptotically behaves as $e^{-f(v)t}$ with lower order fluctuations given by the Tracy-Widom distribution, according to Conjecture \ref{conjecture}. A crossover between this Tracy-Widom distributed fluctuations and the Edwards-Wilkinson fluctuations in the diffusive scaling occurs when $x$ scales as $t^{3/4}$, as Theorem \ref{Thm-convergence} shows. 
\begin{figure}
	\centering
	\usetikzlibrary{intersections, pgfplots.fillbetween}
	\begin{tikzpicture}[scale=0.95, every text node part/.style={align=center}]
		\fill[black!20!] (0,0) -- (8,-8)  -- (8,0) -- cycle; 
		\draw (0,-0.1) -- (0,0.1) node[above] {$0$};
		\fill[black!20!] (0,0) -- (-8,-8)  -- (-8,0) -- cycle; 
		\draw[thick, -stealth] (-8,0) -- (8,0) node[above]{$x$};
		\draw[thick, -stealth] (-8,0) -- (-8,-9) node[right]{$t$};
		\node[]  at (0,-8) {$x=O\left(\sqrt{t}\right)$ \\ \footnotesize Stochastic heat equation};
		\node[draw, color=teal, rounded corners=4pt, line width=2pt] (KPZ) at (5.5,-8) {$x=v t, v\in (0,1)$\\ \footnotesize KPZ fixed point};
		\node[draw, color=teal, rounded corners=4pt, line width=2pt] at (-5,-8) {$x=v t, v\in (-1,0)$\\ \footnotesize KPZ fixed point};
		\node[draw, rounded corners=4pt, dashed, thick ] at (4,-9.5) {\footnotesize$x=t^{3/4}$\\ \footnotesize moderate deviation regime};
		\draw[thick, teal, line width=2pt] (0.3, -0.3) -- (KPZ);
		\node[draw, color=gray, thick,  fill=black!10!, rounded corners=2pt, thick] (frozen) at (6,-3) {$x> t$\\ \footnotesize $\eta(t,x)=0$};
		\node[draw, color=gray, thick,  fill=black!10!, rounded corners=2pt, thick] (frozen) at (-6,-3) {$x< -t$\\ \footnotesize $\eta(t,x)=0$};
		\begin{scope}[rotate=-90]
			\draw[domain=0:8, thick, samples=100, line width=2pt, name path=A] plot (\x, {0.6*sqrt(\x)});
			\draw[domain=0:8, thick , samples=100, line width=2pt] plot (\x, {-0.6*sqrt(\x)});
			\draw[domain=1:9, thick, dashed,  samples=10] plot (\x, {0.6*exp(0.65*ln(\x))});
			\draw[domain=1:9, thick, dashed,  samples=10] plot (\x, {-0.6*exp(0.65*ln(\x))});
		\end{scope} 
	\end{tikzpicture}
	\caption{{\color{black}Diagram in space-time coordinates showing the expected scaling limits of the KMP model with  initial condition $\eta(0,x)=\mathbbm{1}_{0}(x)$. In the gray area, the energy field is frozen (equal to zero). In the diffusive scaling window in the middle, the KMP model was expected
		 to converge to the linear stochastic heat equation. When this article was first posted online, this was only a conjecture,  based on similar results for other models  \cite{balazs2006random,yu2016edwards}. However, the statement should follow from  the results of the recent preprint \cite{drillick2025random} together with   our Proposition \ref{Proposition-KMP-SFK}.  Conjecture \ref{conjecture} is about the scaling $x=vt, v\in (0,1)$, in  green. Our main result (Theorem \ref{Thm-convergence}) shows that between these last two regions, there exist a critical scaling  $x\sim t^{3/4}$ where fluctuations are described by the KPZ equation.}}
	\label{fig:noisyheat}
\end{figure} 

\subsubsection{A sticky limit of the KMP model?}
	Given the coefficient $\frac{1}{\sqrt{{\color{black}2}\alpha}}$ in front of the noise term in \eqref{SHE}, it seems natural to simultaneously rescale time and space with the noise strength $\alpha$. It would be interesting to apply this scaling to the KMP model. More precisely, we ask the following: 
	\begin{question}
		Does the limit 
	\begin{equation}
		E(t,x):= \lim_{\eps\to 0 } \eps^{-1} \eta(\eps^{-2}t, \eps^{-1}x)
		\label{eq:stickylimit}
	\end{equation} 
		exists in a weak sense when the parameter $\alpha$ is scaled {\color{black}as} $\alpha=\eps$. How to describe the limiting Markov process? 
	\end{question}
 In the case of the discrete-time Beta RWRE, the limit is studied in \cite{le2004products,
 	jan2004sticky,howitt2009consistent,le2013markovian,brownianWeb,barraquand2020large,barraquand2022,brockington2023bethe}. Based on \cite{brownianWeb}, we believe that a good candidate for the limit \eqref{eq:stickylimit} should be searched among the  Howitt-Warren flows \cite{howitt2009consistent}. 
	
\subsubsection{More general models of energy redistribution} We expect that, up to model dependent constants, the statement of Theorem \ref{Thm-convergence} still holds for generalizations of the KMP model, with the same proof method. 

 First, one can consider a generalization of the weight's distribution. The most minor generalization would be to choose $\text{Beta}(\alpha,\beta)$ with $\alpha\neq \beta$ random variables. This would simply induce a drift in the density profile of the associated RWRE. More generally, one could consider that the energy is redistributed on a given bond according to a random variable $B$ following a more general distribution, for instance it could be a sum of Dirac masses, or a mixture of Dirac masses and a continuous part. 
 
 Second, one could consider more complicated redistribution mechanism. For example, one could redistribute the energies on a larger number $n$ of consecutive sites, this number $n$ could even be random. As long as the sum of energies is preserved by the redistribution rule, it would still be possible to define a stochastic flow of kernels and possibly apply Theorem \ref{Thm-parekh} as long as some mild assumptions are satisfied. To give a concrete example,  one could consider the following model: associate a Poisson process to each $n$-tuple of consecutive sites and redistribute the energies on these sites at rate $1$ in the same proportions as an independent  $\text{Dirichlet}(\alpha_{1},\ldots,\alpha_{n}))$ random variable.
 
{\color{black}For such more general redistribution laws, the overall proof strategy is expected to remain unchanged. In particular, the verification of the six assumptions of Theorem \ref{Thm-parekh} can be carried out along the same lines. The main additional difficulty may instead concern the explicit identification of the invariant measure of the associated two-point difference process. In the absence of an explicit invariant measure, the computation of the limiting variance, and hence of the noise coefficient, may become less explicit and require additional model-dependent analysis.}

\subsubsection{Convergence in a better topology}{\color{black}In order to prove that some random field converges to the KPZ equation, a standard approach is to prove that the exponential of the field, appropriately rescaled, converges to \eqref{SHE-general} in the space of continuous functions. Then, taking the logarithm, it means that the original field converges to the KPZ equation.  Back to the KMP model, it would be interesting to prove convergences in better topologies.  }

However, one cannot prove the pointwise convergence for the field \eqref{distribution-field}, since $\eta(t,x)$ locally converges to i.i.d. Gamma distributed variables (the stationary measure) which is very irregular at large scales. We expect that replacing $\eta(t,x)$ with $\sum_{y\,:\,y\geq x}\eta(t,y)$ in  \eqref{distribution-field} would make the field converge pointwise (the fluctuations due to the stationary measure would average out in the sum). Such a result is proved in the context of the Beta RWRE  which is a more tractable  model \cite{das2024multiplicative}. We do not pursue this direction here.

\subsection{Notations}\label{section-notation}
It is important in the proofs below to distinguish the probability measure on the environment from the law of the random walk, and to distinguish the averaged law of  the random walk from the random walk conditioned on the environment. We will use the following notations:
 
 \medskip 
 \noindent
\begin{tabular}{|m{0.05\textwidth}|m{0.9\textwidth}|}
	\hline 
 $\mathbb{P}$ &  Probability measure of the environment.\\ \hline 
$\mathbf{P}$ & Path-space annealed probability of the continuous-time random walk. \\ \hline 
$\mathrm{P}$ & Path-space probability measure of a discrete-time symmetric random walk  $(\mathcal{X}_{n})_{n\in \mathbb{Z}_{\geq 0}}$.\\ \hline 
 $\mathrm{P}^{\mathcal{P},\omega}$& Quenched (conditioning on the environment) path-space probability measure for the random walk. Similarly $\mathrm{P}^{\mathcal{P}}$ (conditioning on Poisson processes only).\\ \hline 
 $\mathfrak{P}$  & Law of a  Poisson process  $(N(t))_{t\geq 0}$ with intensity $1$.\\ \hline 
\end{tabular}
\medskip 

The expectations, variances and covariances associated with the probability distributions listed above will be written using the corresponding font.

 \subsection{A general convergence result}\label{section-generalConvergence}
 In \cite{parekh2} the author considered a generalized model of discrete-time random walk in random environment, proving the convergence of the quenched density field to the multiplicative noise stochastic heat equation (mSHE). To our purposes, we only need to rely on the particular case where $p=1$ and the RWRE hops on  the lattice $\mathbb{Z}$. Therefore, we state below the result of \cite{parekh2} in this specific set-up. 
 Consider a probability space $(\Omega,\mathcal{F},\mathbb{P})$. A measurable function $\omega \in \Omega\to \mathcal{K}^{\omega}(\cdot,\cdot)$ such that for each $x\in \mathbb{Z}$ associates a probability measure  $\mathcal{K}^{\omega}(x,\cdot)$ on $\mathbb{Z}$ is said a Markov transition kernel in random environment.  
 Let $\left(\mathcal{K}_{n-1,n}\right)_{n\in \mathbb{Z}_{\geq 0}}$ by a sequence of these kernels (for which we drop the $\omega$ dependence for the sake of notation). For each $n$ we interpret $\mathcal{K}_{n-1,n}(x,y)$ as {\color{black}the} probability of reaching $y$ at time $n$ given that the particle started from $x$ at time $n-1$. 
 Moreover, for all $y\in \mathbb{Z}$ and for all $N\in \mathbb{Z}_{\geq 0}$, define the quenched $N$-steps transition kernel 
 \begin{equation*}
 	\mathcal K_{0,N}:=\mathcal{K}_{0,1}\cdots \mathcal{K}_{N-1,N}\,,
 \end{equation*}
 where $\mathcal{K}_{0,1}\mathcal{K}_{1,2}(y,x)=\sum_{z\in \mathbb{Z}}\mathcal{K}_{0,1}(y,z)\mathcal{K}_{1,2}(z,x)$. We assume that the sequence satisfies the following hypotheses.
 
 We introduce some notation. We define the $1$-point annealed kernel as  \begin{equation}\label{one-pts-kernel-general}
 	\mathbf p(x):=\mathbb{E}\left[\mathcal{K}_{1}(0,x)\right].
 \end{equation} 
 Then, we introduce its moment generating function $M(\alpha):=\sum_{x\in \mathbb{Z}}e^{\alpha|x|}\mathbf p(x)<\infty$ and its $k$-th moments as $m_{k}=\sum_{x\in \mathbb{Z}}x^{k}\mathbf p(x)$. 
 Moreover, considering  $\bm{x}=(x_{1},\ldots,x_{k})$ and $\bm{y}=(y_{1},\ldots,y_{k})$ with $x_{i},y_{i}\in \mathbb{Z}$ and define the $k$-points annealed correlation kernel
 \begin{equation*}
 	\mathbf{p}^{(k)}\left((x_{1},\ldots,x_{k})\to (y_{1},\ldots,y_{k})\right):=\mathbb{E}\left[\mathcal{K}_{1}(x_{1},y_{1})\cdots \mathcal{K}_{1}(x_{k},y_{k})\right]\,.
 \end{equation*}
 Finally, by marginalizing over the $2$-points kernel, we introduce the $2$-points kernel
 \begin{equation}\label{p-dif}
 	\mathbf{p}_{\text{dif}}(x,y):=\sum_{y_{1},y_{2}\in\mathbb{Z}}\mathbbm{1}_{\{y_{1}-y_{2}=y\}}\mathbf{p}^{(2)}\left((x,0)\to(y_{1},y_{2})\right)
 \end{equation}
 We denote by  $(Z(n))_{n\in\mathbb{Z}}$ the Markov chain associated to the transition kernel $\mathbf{p}_{\text{dif}}(x,y)$.
 \begin{assumption}\label{assumption}
 	We assume the following 6 statements: 
 	\begin{itemize}
 		\item[\textbf{HP 1)}] $\mathcal{K}_{0,1}(\cdot,\cdot),\mathcal{K}_{1,2}(\cdot,\cdot),\mathcal{K}_{2,3}(\cdot,\cdot)\ldots$ are i.i.d.  under $\mathbb{P}$.
 		\item[\textbf{HP 2)}] The kernels $\mathcal{K}_{1}(x,A)$ and $\mathcal{K}_{1}(x+a,A+a)$ have the same distribution under $\mathbb{P}$.
 		\item[\textbf{HP 3)}] 
 		There exists $\alpha>0$ such that $M(\alpha)<\infty$ and  $m_{2}-m_{1}^2=1$.
 		\item[\textbf{HP 4)}] The first moment is non-deterministic :  $\mathbb{VAR}\left(\int_{\mathbb{Z}}y\,\mathcal{K}_{1}(0,dy)\right)>0$.
 		
 		\item[\textbf{HP 5)}] 
 		We can construct a decreasing function $F_{\text{decay}}:[0,\infty)\to [0,\infty)$ such that  $x\to xF_{\text{decay}}(x)$ is decreasing for $x$ large enough and  belongs to $L^{1}([0,\infty),dx)$, and such that the following estimate holds for all $k\leq 4$, $r_{1},\ldots,r_{k}\in \{0,\ldots,3\}$, $r_{1}+\ldots+r_{k}\leq 4$.
 		\begin{equation*}
 			\left|\sum_{y_{1},\ldots,y_{k}\in \mathbb{Z}}\prod_{j=1}^{k}(y_{j}-x_{j})^{r_{j}}\mathbf{p}^{(k)}(\mathbf{x}\to\mathbf{y})-\sum_{y_{1},\ldots,y_{k}\in \mathbb{Z}}\prod_{j=1}^{k}y_{j}^{r_{j}}\mathbf{p}^{\otimes k}(\mathbf{y})\right| 
 			\leq F_{\text{decay}}\left(\min_{1\leq i<j\leq k}|x_{i}-x_{j}|\right)\,.
 		\end{equation*}
 		\item[\textbf{HP 6)}] 
 		The Markov chain $(Z(n))_{n\in \mathbb{Z}_{\geq 0}}$,  is irreducible:  $\forall x,y\in\mathbb{Z}$ there exists $m\in \mathbb{Z}_{\geq 0}$ such that $\mathbf{p}_{\text{dif}}^{m}(x,y)>0$.
 	\end{itemize}
 \end{assumption}
 
 \begin{remark} In words, \textbf{HP 5)} requires that, when the $x_{i}$ are far apart, the $k$-points motion is close to the motion of $k$ independent random walkers. This statement must hold, at least, in the sense of mixed moments. 
 \end{remark}
 
 We introduce some quantities. For all $N\in \mathbb{Z}_{\geq 0}$, we define the drift constant 
 \begin{equation*}
 	d_{N}:=N\sum_{x\in \mathbb{Z}}xe^{xN^{-1/4}-\log M\left(N^{-1/4}\right)}\mathbf{p}(x).
 \end{equation*}
 Furthermore, for $t\in N^{-1}\mathbb{Z}_{\geq 0}$ and for $x\in \mathbb{Z}$, we introduce the tilting constant
 \begin{equation}\label{tilting-constant}
 	D_{N,t,x}:=\exp{\left\{N^{1/4}x+t\left[N^{-1/4}d_{N}-N\log M(N^{-1/4})\right]\right\}}\,.
 \end{equation}
 The re-scaled field, of which we want to take the limit, is given by (for $t=k/N$ and $\phi\in C_{c}^{\infty}(\mathbb{R})$)
 \begin{equation}\label{field-general}
 	\mathfrak{h}^{N}\left(t,\phi\right):=\sum_{x\in\mathbb{Z}}D_{N,t,N^{-1/2}(x-d_{N}t)}\phi\left(N^{-1/2}\left(x-d_{N}t\right)\right)\mathcal K_{0,tN}(0,x).
 \end{equation}
 The field is defined for integer times in \eqref{field-general} and we can linearly interpolate to define it for all times $t\geqslant 0$. 
 
 In order to provide intuition behind this scaling, let us examine the first moment -- see \cite{parekh2} for further explanations. Let $(\mathcal{R}(n))_{n\in\mathbb{Z}_{\geq 0}}$ denote the random walk associated with the transition kernel $(\mathcal{K}_{n-1,n})_{n\in\mathbb{Z}_{\geq 0}}$. We observe that the stochastic process 
 \begin{equation*}
 	\mathcal Z_n(\lambda) =	e^{\lambda\mathcal{R}(n)-n\log{M(\lambda)}}=\frac{e^{\lambda \mathcal{R}(n)}}{\mathbf{E}[e^{\lambda \mathcal{R}(n)}]}
 \end{equation*}
 is a strictly positive, mean-one discrete-time martingale. Let us fix a horizon time $N>0$ and consider the measure  $\mathbf Q_{\lambda, N}(\cdot ) = \frac{d\mathbf Q_{\lambda, N}}{d\mathbf P} \mathbf{P}(\cdot)$ where $\frac{d\mathbf Q_{\lambda, N}}{d\mathbf P} = \mathcal Z_N(\lambda)$. Since we are interested in moderate deviations where $\mathcal R(n)$ is of order $O(N^{3/4})$, it is natural to let $\lambda=N^{-1/4}$. Then, 
 \begin{equation*}
 	\mathbb{E}\left[\mathfrak{h}^{N}\left(t,\phi\right)\right]=\mathbf Q_{N^{-1/4}, N}\left[\phi\left(N^{-1/2}\left(\mathcal{R}(tN)-d_{N}t\right)\right)\right].
 \end{equation*}
 By Girsanov's theorem, under $\mathbf Q_{N^{-1/4}, N}$, the sequence of random variables $N^{-1/2}\left(\mathcal{R}(tN)-d_{N}t\right)$ is a centered process, and by Donsker's theorem, it weakly converges to a Brownian motion. It implies that
 \begin{equation*}
 	\mathbb{E}\left[\mathfrak{h}^{N}\left(t,\phi\right)\right] \xrightarrow[N\to\infty]{} \int_{\mathbb R} p_t(x) \phi(x)  = \mathbb{E}\left[\int_{\mathbb R} dx\,  \mathcal U_t(x) \phi(x)\right], 
 \end{equation*}
 where $p_t(x)$ is the standard heat kernel and  $\mathcal{U}_{t}(x)$ solves \eqref{SHE-general} with $\delta_{0}$ initial condition.

 \begin{theorem}[{Special case of \cite[Theorem 1.4]{parekh2}}]\label{Thm-parekh}
 	The sequence  $(\mathfrak{h}^{N})_{N\in\mathbb{Z}_{\geq 0}}$ defined in \eqref{field-general} is tight with respect to the topology $C\left([0,T],\mathcal{S}^{'}(\mathbb{R})\right)$. Furthermore, any limit point as $N\to \infty$ lies on $C\left((0,T],C(\mathbb{R})\right)$ and coincides with the law of the It\=o solution to the multiplicative noise stochastic heat equation given by  
 	\begin{equation}\label{SHE-parekh}
 		\begin{cases}
 			\partial_{t}\mathcal{Z}(t,x)=\frac{1}{2}\partial_{xx}\mathcal{Z}(t,x)+\gamma \;\mathcal{Z}(t,x)\Xi(t,x)\\
 			\mathcal{Z}(0,x)=\delta_{0}(x)
 		\end{cases}\,.
 	\end{equation}
 	Here, $\Xi$ is a standard white noise and the noise variance is given by 
 	\begin{equation}\label{constan-general}
 		\gamma^{2}:=\frac{\frac{1}{2}\sum_{z\in\mathbb{Z}}\left[\sum_{x,y\in \mathbb{Z}}(x-y)^{2}\mathbf{p}(x)\mathbf{p}(y)-\sum_{a\in \mathbb{Z}}(a-z)^{2}\mathbf{p}_{\text{dif}}(z,a)\right]\pi^{\text{inv}}(z)}{\sum_{z\in \mathbb{Z}}\left[\sum_{a\in \mathbb{Z}}{\color{black}\mathbf{p}_{\text{dif}}(z,a)}|a|-|z|\right]\pi^{\text{inv}}(z)}\,.
 	\end{equation}
 \end{theorem}
 Above $\pi^{\text{inv}}$ is the invariant measure of the process Markov $(Z(n))_{n\in\mathbb{Z}_{\geq 0}}$, with transition kernel $\mathbf{p}_{\text{dif}}$ introduced in \eqref{p-dif}. It is shown in \cite{parekh2} that, under Assumptions \ref{assumption}, $\pi^{\text{inv}}$ is unique and $0\leq \gamma<\infty$. 
 
 \begin{remark}
 	Theorem \ref{Thm-parekh} is stated for discrete time RWRE, corresponding to discrete stochastic flows of kernels. A similar result should hold for continuous time RWRE, and it would be interesting to understand how $\gamma$ depends on the law of the RWRE in that case. Our arguments in Section \ref{Section-noiceVariance} suggest that for continuous time RWRE $\mathcal R(t)$ on  countable state space, the value of the noise strength $\gamma^2$ in the limit should be the same as for the skeleton chain $\widetilde{\mathcal R}(n)$ associated to $\mathcal R(t)$ (if $\mathcal R(t)$ has rates $q_{x\to y}$, then we define the skeleton chain $\widetilde{\mathcal R}$ as the discrete time Markov chain with transition probabilities $p_{x\to y}= \frac{q_{x\to y}}{-q_{xx}}$).
 \end{remark}
 
\section{Verification of hypotheses}\label{section-assumptions}
For the RWRE from Definition \ref{definition-RWRE}, the one-point annealed law is a continuous time simple random walk, so that 
\begin{equation}\label{mu}
	\mathbf p(y):=\mathbb{E}\left[K_{1}(0,y)\right]=\sum_{n=0}^{\infty}\mathrm{P}^{\mathcal{P}}\left(\sum_{i=1}^{n}\xi_{i}=y\bigg| N(1)=n\right)\mathfrak{P}(N(1)=n)\,,
\end{equation}
where $(N(t))_{t\geq 0}$ denotes a Poisson process (describing when the particle jumps) and  $(\xi_{i})_{i\in \mathbb{Z}_{\geq 0}}$ is a sequence of i.i.d. Rademacher random variables. We may compute the moment generating function as
\begin{align}
		M(\lambda) 
		&=\sum_{n=0}^{\infty}\left(\mathrm{E}\left[e^{\lambda\xi_{1}}\right]\right)^{n}\mathfrak{P}\left(N(1)=n\right) \nonumber\\
		&=e^{\cosh(\lambda)-1}\,. \label{moment-generating-function}
\end{align}
Moreover, using 
\begin{equation*}
		\sum_{x\in\mathbb{Z}}x e^{x N^{-1/4}}\mathbf p(x)=
		\sinh(N^{-1/4})e^{\cosh(N^{-1/4})-1}\,,
\end{equation*}
we find the expression of the drift constant
	\begin{equation}\label{drift-explicit}
		d_{N}=N\sinh(N^{-1/4})\,.
\end{equation}
Given the scalings \eqref{tilting-constant} and \eqref{field-general} in Theorem \ref{Thm-parekh},  the {\color{black}constant} $C_{N,t,x}$ arising in our scalings  \eqref{filed-explicit} and \eqref{skorokhod-field} should be equal to
\begin{equation*}
	{C_{N,t,x} = \color{black}D_{N, t,x+ \frac{ N^{3/4}t - d_N t}{\sqrt{N}}}= \exp{\left(N^{1/4}x+\frac{1}{2}N^{1/2}t-\frac{t}{24}+O(tN^{-1/2})\right)}}
\end{equation*}

{\color{black}Since we are considering a convergence in $C([0,T],\mathcal{S}'(\R))$ we can discard the error term $t O(N^{-1/2})$, hence the expression given in \eqref{filed-explicit}.  Indeed, this definition ensures that $ \langle \mathfrak F_N(t, \cdot), \phi\rangle  = e^{O(tN^{-1/2})}  \mathfrak h^N(t,\widetilde \phi)$ with $ \phi(x)=\widetilde \phi(x+ N^{-1/2}t(N^{3/4}-d_N))$, so that for any $\phi\in C_c^{\infty}(\R)$, 
	$$ \sup_{t\in [0,T]}\mathbb E\left[ \left\vert\langle \mathfrak F_N(t, \cdot), \phi\rangle-\mathfrak h^N(t, \phi) \right\vert\right] \xrightarrow[N\to\infty]{}0.$$}

 {\color{black}Consider the sequence of kernels $(K_{n-1,n}(x,y))_{n\in \mathbb{Z}_{\geq 0}}$ obtained by evaluating the law of the RWRE \eqref{definition-kernel} at integer times $s=n-1$ and $t=n$. }
\begin{proposition}\label{proposition-assumptions-verify}
	Assumptions \ref{assumption} are verified by the sequence of kernels $(K_{n-1,n}(x,y))_{n\in \mathbb{Z}_{\geq 0}}$.  
\end{proposition}
\begin{proof}
Assumptions \textbf{HP 1)}, \textbf{HP 2)} and \textbf{HP 6)} immediately follows from the definition of the sequence of kernels.  Assumption \textbf{HP 3)} is a consequence of the explicit expression for  the moment generating function \eqref{moment-generating-function} and of the fact that $\mathbf{p}$ is the probability distribution of a continuous-time random walk at time $1$ so that $m_1=0$ and $m_2=1$. 

To prove assumption \textbf{HP 4)} we simply need to prove that $	\sum_{y\in \mathbb{Z}}yK_{1}(0,y)$ has some probability to be nonzero. Indeed, we may assume that there are no Poisson events associated with the bonds $(-1,0)$ and $(1,2)$ and only one at time $0<T<1$ on the bond $(0,1)$. This happens with positive probability (with respect to the law of independent Poisson processes), and in such a case, $	\sum_{y\in \mathbb{Z}}yK_{1}(0,y)= {\color{black}1-B_{T}}$, which is nonzero with positive probability. 

Assumption \textbf{HP 5)} is trivial if $k=1$ or if $r_{1}+\ldots+r_{k}=1$. First, we prove it for $k=2$ and we generalize for arbitrary $k$. 
{\color{black}Let $(X_1(t))_{t\geq 0}$ and $(X_2(t))_{t\geq 0}$ be two RWRE, as in Section \ref{subsection-RWRE}, initialized at $t=0$ from
	$x_1,x_2\in\mathbb{Z}$ and conditionally independent given the environment.
	Their joint annealed evolution is called the two-point motion and has
	transition kernel
	\begin{equation*}
		\mathbf{p}^{(2)}_t((x_1,x_2)\to(y_1,y_2))
		:=
		\mathbb{E}\left[
		K_t(x_1,y_1)K_t(x_2,y_2)
		\right].
	\end{equation*}
	The walkers evolve independently until they "share" the same Poisson
	event; in that case, the corresponding joint transition probabilities involve the second moment of the Beta variable associated with that bond.
 Assume that at $t=0$ the initial locations of the two walkers satisfy } $|x_{1}-x_{2}|=2d$ for some $d\in \frac{\mathbb{Z}_{\geq 0}}{2}$ and, without loss of generality, we assume $x_{1}\leq x_{2}$.  We introduce the Poisson event $\mathcal{E}_{d}$ defined as follows: 
\begin{equation*}
	\mathcal{E}_{d}:=\left\{\exists t \in [0,1]\;:\;X_{1}(t)=\lfloor\frac{x_{1}+x_{2}}{2}\rfloor\;\vee\;X_{2}(t)=\lceil \frac{x_{1}+x_{2}}{2}\rceil\right\}\,.
\end{equation*}

{\color{black}
\begin{equation*}
	\mathcal{E}_{d}=\mathcal{E}_{d}^{(1)}\cup \mathcal{E}_{d}^{(2)}\cup \mathcal{E}_{d}^{(1,2)}
\end{equation*} 
where 
\begin{equation*}
	\mathcal{E}_{d}^{(1)}:= \left\{\exists t\in [0,1]\,:\, X_{1}(t)=\lfloor\frac{x_{1}+x_{2}}{2}\rfloor\quad \&\quad \forall t\in [0,1]\,; X_{2}(t)\neq \lceil \frac{x_{1}+x_{2}}{2}\rceil\right\}
\end{equation*}
\begin{equation*}
	\mathcal{E}_{d}^{(2)}:= \left\{\forall t\in [0,1],\; X_{1}(t)\neq \lfloor\frac{x_{1}+x_{2}}{2}\rfloor\quad \&\quad \exists t\in [0,1]\; X_{2}(t)= \lceil \frac{x_{1}+x_{2}}{2}\rceil\right\}
\end{equation*}
\begin{equation*}
	\mathcal{E}_{d}^{(1,2)}:= \left\{\exists t\in [0,1]\,:\, X_{1}(t)=\lfloor\frac{x_{1}+x_{2}}{2}\rfloor\quad \&\quad \exists  s\in [0,1]\,:\, X_{2}(s)= \lceil \frac{x_{1}+x_{2}}{2}\rceil\right\}
\end{equation*}
}

For $i=1,2$, for fixed $x_{i}\in \mathbb{Z}$ and $ r_{i} \in\{0,1,2,3\}$ such that $2\leq r_{1}+r_{2}\leq 4$, we introduce $\Delta_{i}:=(X_{i}(1)-x_{i})^{r_{i}}$ and we write that 
{\color{black}
\begin{equation*}
	\begin{split}	\mathbf{E}\left[\Delta_{1}\Delta_{2}\right]=&\mathbf{E}\left[\Delta_{1}\Delta_{2}|\mathcal{E}_{d}^{(1)}\right]\mathfrak{P}(\mathcal{E}_{d}^{(1)})+\mathbf{E}\left[\Delta_{1}\Delta_{2}|\mathcal{E}_{d}^{(2)}\right]\mathfrak{P}(\mathcal{E}_{d}^{(2)})
		\\+&
		\mathbf{E}\left[\Delta_{1}\Delta_{2}|\mathcal{E}_{d}^{(1,2)}\right]\mathfrak{P}(\mathcal{E}_{d}^{(1,2)})+\mathbf{E}\left[\Delta_{1}\Delta_{2}|\mathcal{E}_{d}^{c}\right]\mathfrak{P}(\mathcal{E}_{d}^{c})\,.
	\end{split}
\end{equation*}
Consider the first addend above, we have that 
\begin{equation*}
	\mathbf{E}[\Delta_{1}\Delta_{2}|\mathcal{E}_{d}^{(1)}]\leq \sqrt{\mathbf{E}[\Delta_{1}^{2}|\mathcal{E}_{d}^{(1)}]\,\mathbf{E}[\Delta_{2}^{2}|\mathcal{E}_{d}^{(1)}]}\,.
\end{equation*}
For $i\in\{1,2\}$, we denote by $\Delta_{i}(n)$ the displacement of the $i$-th walker in $n$ jumps. Since the increment have polynomial growth and by using the Poisson distribution formula,  we have the bound 
\begin{equation*}
	\mathbf{E}\left[\Delta_{1}^{2}|\mathcal{E}_{d}^{(1)}\right]\mathfrak{P}(\mathcal{E}_{d}^{(1)})
	=\sum_{n=\lfloor d\rfloor }^{\infty} \mathrm{E}[\Delta_{1}^{2}(n)]\mathfrak{P}(n,\mathcal{E}_{d}^{(1)})\leq C_{1} e^{-C_{2} d }\,,
\end{equation*}
where $C_{1},C_{2}>0$ where $\mathfrak{P}(n,\mathcal{E}_{d}^{(1)})$ is the probability that there are $n$ points of the Poisson process on interval $[0,1]$ and $\mathcal{E}_{d}^{(1)}$ holds.
Moreover, for the second walker we have that 
\begin{equation*}
	\mathbf{E}\left[\Delta_{2}^{2}|\mathcal{E}_{d}^{(1)}\right]\mathfrak{P}(\mathcal{E}_{d}^{(1)})= \sum_{n=0 }^{\infty} \mathrm{E}[\Delta_{2}^{2}(n)]\mathfrak{P}(n,\mathcal{E}_{d}^{(1)})\leq C_{3}\,,
\end{equation*}
where $C_{3}>0$.
For the events $\mathcal{E}_{d}^{(2)}$ and $\mathcal{E}_{d}^{(1,2)}$ the argument is similar and produces exponential upper bounds. The last addend $\mathbf{E}\left[\Delta_{1}\Delta_{2}|\mathcal{E}_{d}^{c}\right]\mathfrak{P}(\mathcal{E}_{d}^{c})$ corresponds to the motion of two independent walkers. Moreover, by the argument above,  $\mathfrak{P}(\mathcal{E}_{d}^{c})$ equal $1$ minus an exponential term, decaying to $0$ in $d$. Hence, up to an exponentially small error, this last added can be replaced by  $\sum_{y_{1},y_{2}\in \mathbb{Z}}\prod_{j=1}^{2}y_{j}^{r_{j}}\mathbf{p}^{\otimes 2}(y_{j})$. 
}

 All in all, we conclude that, for all $r_{1},r_{2}\in \{0,\ldots,3\}\,:\, 2\leq r_{1}+r_{2}\leq 4$, there exist $C_{3},C_{4}>0$ such that

\begin{equation*}
	\begin{split}
		&\left|\sum_{y_{1},y_{2}\in \mathbb{Z}}\prod_{j=1}^{2}(y_{j}-x_{j})^{r_{j}}\mathbf{p}^{(2)}((x_{1},x_{2})\to(y_{1},y_{2}))-\sum_{y_{1},y_{2}\in \mathbb{Z}}\prod_{j=1}^{2}y_{j}^{r_{j}}\mathbf{p}^{\otimes 2}(y_{j})\right|\leq C_{3}e^{-C_{4}|x_{1}-x_{2}|}\,.
	\end{split}
\end{equation*}
To extend the argument to general $k\in \{3,4\}$, we introduce the Poisson event by assuming that the walkers are labelled according to $x_i\leq x_{i+1}$ for $1\leq i\leq k-1$:
\begin{equation*}
	\mathcal{E}:=\bigcup_{1\leq i\leq k} \bigcup_{t\in[0,1]}  \left\lbrace  X_{i}(t)=\;\left\lceil \frac{x_{i-1}+x_{i}}{2} \right\rceil \right\rbrace \cup\left\lbrace X_{i}(t)=\left\lfloor \frac{x_{i}+x_{i+1}}{2}\right\rfloor\right\rbrace.
\end{equation*}
Then, by applying repeatedly the Cauchy-Schwarz inequality, {\color{black} we obtain an exponential decay in the variable $2d=\min_{1\leq i\leq k-1}(x_{i+1}-x_{i})$.} 
\end{proof}
We introduce the sequence of fields $$\left(\bar{\mathfrak{F}}_{N}(t,x)\right)_{N\in \mathbb{Z}_{\geq 0}}\subset C([0,T],\mathcal{S}(\mathbb{R})),$$ 
where $\bar{\mathfrak{F}}_{N}(t,x)$ is obtained by linearly interpolating the {\color{black}field} $\mathfrak{F}_{N}(t,x)$  between the points in $\left\{\frac{
k}{N}T\right\}_{k=0}^{N}$.
Applying Theorem \ref{Thm-parekh}, we conclude that the sequence of field  $\left(\bar{\mathfrak{F}}_{N}(t,x)\right)_{N\in \mathbb{Z}_{\geq 0}}$
 is tight in $C([0,T],\mathcal{S}'(\mathbb{R}))$ and it weakly converges to the solution $\mathcal{U}(t,x)$ of \eqref{SHE} in the sense of Theorem \ref{Thm-parekh}. 
\section{Convergence in the Skorokhod space}\label{section-skorokhod}
Building on the convergence of the sequence of interpolated fields $\left(\bar{\mathfrak{F}}_{N}(t,x)\right)_{N\in \mathbb{Z}{\geq 0}}$ to the solution $\mathcal{U}(t,x)$ of \eqref{SHE}, established in the previous section, we now prove that the sequence of fields $\left(\mathfrak{F}_{N}(t,x)\right)_{N\in \mathbb{Z}{\geq 0}}$, introduced in \eqref{skorokhod-field}, likewise converges in the Skorokhod space $D([0,T],\mathcal{S}'(\mathbb{R}))$.
\begin{proposition}\label{proposition-skorokhod}
The sequence of fields $\left(\mathfrak{F}_{N}(t,x)\right)_{N\in \mathbb{Z}_{\geq 0}}$ is tight in $D([0,T],\mathcal{S}
'(\mathbb{R}))$ and its finite dimensional distributions weakly converge to the finite dimensional distributions of $\mathcal{U}(t,x)$, solution of \eqref{SHE}.  
\end{proposition}
\begin{proof} 
First, we show tightness. Referring to \cite{mitoma}, it is enough to prove tightness in $D([0,T],\mathbb{R})$ for the sequence of processes $(\left\langle \mathfrak{F}_{N}\left(t,\cdot\right),\phi\right\rangle)_{N\in \mathbb{Z}_{\geq 0}} $ defined in \eqref{skorokhod-field}. To this aim we state and prove the following Lemma. 
\begin{lemma}\label{Lemma5-2}
	For all $\phi\in C_{c}^{\infty}$  there exists a constant $C>0$  such that  for all  $t,s\in[0,T]$ with $s<t$,
	\begin{equation}\label{estimate}
			{\color{black}\mathbb{E}\left[\left|\left\langle \mathfrak{F}_{N}\left(t,\cdot\right),\phi\right\rangle-\left\langle \mathfrak{F}_{N}\left(s,\cdot\right),\phi\right\rangle\right|\right]\leq C (t-s)^{1/4}}\,.
	\end{equation}
\end{lemma}

\begin{proof}
{\color{black}	Let $(X_{N}(t))_{t\geq 0}$ be the RWRE defined in Section \ref{subsection-RWRE}. Then, we have that  the continuous time stochastic process
	\begin{equation}\label{q-definition}
q(t,X_{N}(Nt)):= \frac{e^{N^{-1/4}X_{N}(Nt)}}{\mathbf{E}[e^{N^{-1/4}X_{N}(Nt)}]}\,.
	\end{equation}
	is a mean-one continuous-time martingale. We will interpret this martingale as a Radon-Nikodym derivative and we further have that this coincides with the tilting constant \eqref{tilting-constant} after a suitable shift of the space variable, that is
	\begin{equation*}
			D_{N,t,N^{-1/2}(X_{N}(Nt)-td_{N})}=q(t,X_{N}(Nt))\,.
			\end{equation*} 
			Let $N\in \mathbb{Z}_{\geq 0}$ and consider the path space measure for the random walk $(X_{N}(t))_{t\geq 0}$ given by $\mathbf{Q}_{N}(\cdot)=\frac{d \mathbf{Q}}{d\mathbf{P}}\mathbf{P}(\cdot)$, where $\frac{d\mathbf{Q}_{N}}{d\mathbf{P}_{N}}=q(t,X_{N}(Nt))$. Then, by a direct computation, one has that 
				\begin{equation*}
						\mathbf{Q}_{N}\left(X_{N}(Nt)\right)=d_{N}t\,.
					\end{equation*}
					We conclude the stochastic process $(Y_{N}(Nt))_{t\geq 0}$, defined as $Y_{N}(Nt):=X_{N}(Nt)-d_{N}t$, has zero expectation under $\mathbf{Q}_{N}$.
The aim is to estimate 
					\begin{equation*}
						\mathbb{E}\left[\left|\mathrm{E}^{\mathcal{P},\omega}\left[q(t,X_{N}(Nt))\phi\left(\frac{X_{N}(Nt)-d_{N}t}{\sqrt{N}}\right)-q(s,X_{N}(Ns))\phi\left(\frac{X_{N}(Ns)-d_{N}s}{\sqrt{N}}\right)\right]\right|\right]\,.
					\end{equation*}
					By using the Chapman-Kolmogorov property and the triangle inequality we  obtain
					\begin{multline*}
						\mathbb{E}\left[	\left|\mathrm{E}^{\mathcal{P},\omega}\left[q(t,X_{N}(Nt))\phi\left(\frac{X_{N}(Nt)-d_{N}t}{\sqrt{N}}\right)-q(s,X_{N}(Ns))\phi\left(\frac{X_{N}(Ns)-d_{N}s}{\sqrt{N}}\right)\right]\right|\right]
							\\=
							\mathbb{E}\left[\left|\sum_{x}K_{0,Nt}(0,x)q(t,x)\phi\left(\frac{x-d_{N}t}{\sqrt{N}}\right)-\sum_{y}K_{0,Ns}(0,y)q(s,y)\phi\left(\frac{y-d_{N}s}{\sqrt{N}}\right)\right|\right]
							\\\leq
							\mathbb{E}\left[\sum_{x,y}K_{0,Ns}(0,y)K_{Ns,Nt}(y,x)q(t,x)\left|\phi\left(\frac{x-d_{N}t}{\sqrt{N}}\right)-\phi\left(\frac{y-d_{N}s}{\sqrt{N}}\right)\right|\right]
							\\+
							\mathbb{E}\left[\left|\sum_{x,y}K_{0,Ns}(0,y)K_{Ns,Nt}(y,x)\phi\left(\frac{y-d_{N}s}{\sqrt{N}}\right)q(t,x)-q(s,y)\right|\right]
							\\=
							\mathrm{E}_{\mathbf Q_{N}}\left[\left|\phi\left(\frac{X_{N}(Nt)-d_{N}t}{\sqrt{N}}\right)-\phi\left(\frac{X_{N}(Ns)-d_{N}s}{\sqrt{N}}\right)\right|\right]
							\\ +
							\mathbb{E}\left[\left|\mathrm{E}^{\mathcal{P},\omega}\left[\phi\left(\frac{X_{N}(Ns)-sd_{N}}{\sqrt{N}}\right)\left(q(t,X_{N}(Nt))-q(s,X_{N}(Ns))\right)\right]\right|\right]\,.
					\end{multline*}
					
					By Lipschitz continuity of $\phi$ and by 
					\begin{align*} \mathbf{Q}_{N}\left(\left|\frac{Y_{N}(tN)}{\sqrt{N}}-\frac{Y_{N}(sN)}{\sqrt{N}}\right|\right) &\leq \sqrt{ \frac{1}{N} \mathbf{Q}_{N}\left(\left(Y_{N}((t-s)N)\right)^2\right) }\\
						&\leq \sqrt{(t-s) \cosh(N^{-1/4})}\,,
						 \end{align*}
						 we obtain that 
						 $$
						 \mathrm{E}_{\mathbf Q_{N}}\left[\left|\phi\left(\frac{X_{N}(Nt)-d_{N}t}{\sqrt{N}}\right)-\phi\left(\frac{X_{N}(Ns)-d_{N}s}{\sqrt{N}}\right)\right|\right]\leq C(t-s)^{1/2}
						 $$
					for some $C>0$,	providing an estimate for the first addend. 
					
					To estimate the second addend we introduce 
					\begin{equation*}
						Z_{t}^{(N)}:= \sum_{x}K_{0,Nt}(0,x)q(t,x)\,,
					\end{equation*}
					which has  the property that $\mathbb{E}\left[Z_{t}^{(N)}\right]=1$. 
					Using the translation invariance and time homogeneity of the quenched kernel $K_{s,t}(y,x)$, as well as the Markov property, we have that 
					\begin{multline*}	\mathbb{E}\left[\left|\sum_{y}K_{0,Ns}(0,y)\phi\left(\frac{y-d_{N}s}{\sqrt{N}}\right)\left(\sum_{x}K_{Ns,Nt}(y,x)(q(t,x)-q(s,y))\right)\right|\right]
							\\\leq 	
							\lVert \phi\rVert_{\infty}\mathbb{E}\left[	\sum_{y}K_{0,Ns}(0,y)q(s,y)\left|\left(\sum_{x}K_{Ns,Nt}(0,x-y)(q(t-s,x-y)-1)\right)\right|\right]
							\\= \lVert \phi\rVert_{\infty}	\mathbb{E}\left[\left|Z_{t-s}^{(N)}-1\right|\right]\,.
					\end{multline*}
					By Cauchy-Schwarz we obtain
					\begin{equation*}
							\mathbb{E}\left[\left|Z_{t-s}^{(N)}-1\right|\right]\leq \sqrt{\mathbb{E}\left[\left(Z_{t-s}^{(N)}-1\right)^{2}\right]}\,.
					\end{equation*}
					Furthermore, one has that 
					\begin{equation*}
						\mathbb{E}\left[(Z_{t-s}^{(N)}-1)^{2}\right]=\mathbb{E}\left[(Z_{t-s}^{(N)})^{2}\right]-1\,.
					\end{equation*}
					Therefore, it is enough to find an estimate for $\mathbb{E}\left[(Z_{t-s}^{(N)})^{2}\right]$. In particular, we estimate the second moment of $Z_{t}^{N}$ using the local time of the difference walk in $x=0$. We denote $\bm{x}=(x_{1},x_{2})$ and $\bm{X}(r)=(X_{1}(r),X_{2}(r))$ for $r\geq 0$ and we introduce
					\begin{equation*}
						h_{N}(\bm{x})=e^{N^{-1/4}(x_{1}+x_{2})}\,.
					\end{equation*}
					We want to estimate the following expectation 
					\begin{equation*}
						u(r,\bm{x})=	\mathbf{E}_{\bm{x}}\left[e^{N^{-1/4}(X_{1}(r)+X_{2}(r))-c_{N}r }\right]\,,
					\end{equation*}
					with $c_{N}=2\cosh(N^{-1/4})-2$.
					The rest of the proof is quite technical, therefore we report here only the main steps and results, postponing details  to Appendix \ref{Appendix-details5-2}. 
					
					Denote by $L^{(2)}$ the generator of the couple $(X_1,X_2)$ and by $L^{\rm ind}$ the generator of two independent walks. We observe that $h_N$ is a harmonic function with respect to  $L^{\rm ind}-c_N$. We introduce a sort of Doob-transformed operator $\mathcal A_N = h_N^{-1} (L^{(2)}-c_N) h_N$ such that for $v(t,\mathbf x) = h_N^{-1}(\mathbf x) u(t,\mathbf x)$ we have   $\partial_t v(t,\mathbf x) = \mathcal A_N v(t,\mathbf x)$, and we may write 
					$$ \mathcal A_N = \widehat{L}_N + V_N\,, $$
					where $V_N$ is some explicit potential term such that $V_N(\mathbf x) \leq N^{-1/2} C_{\alpha} \mathbbm{1}_{x_1=x_2}$, $\widehat{L}_N$ is the generator of a continuous-time Markov process on $\mathbb Z^2$, that we denote by $(\widehat Y_1(t), \widehat Y_2(t))$. 
					Denoting the difference $D_t= \widehat Y_1(t)- \widehat Y_2(t)$, we obtain that 
					$$ u(t, \mathbf x) \leq \widehat{\mathbf{E}}_{\bm{x}}\left[\exp{\left(C_{\alpha}N^{-1/2}\int_{0}^{r}\mathbbm{1}_{\{D_{s}=0\}}ds\right)}\right].$$ Above we denoted by $\widehat{\mathbf{E}}_{x}$ the expectation with respect to the law of $\hat{L}_{N}$. 
					One can explicitly compute the transition rates of $D$ in terms of constants 
					\begin{equation*}
						a= \mathbb{E}[B(1-B)]=\frac{\alpha}{2(2\alpha+1)},\qquad b= \mathbb{E}[B^{2}]=\frac{\alpha+1}{2(2\alpha+1)}\,.
					\end{equation*}
					The rates are:
					\begin{equation*}
						\begin{array}{c|c}
							\text{transition}&\text{rate}\\ \hline
							d\to d\pm1,\quad |d|\ge2&\cosh(N^{-1/4})\\
							0\to\pm1&2a\cosh(N^{-1/4})\\
							\pm1\to0&2b\cosh(N^{-1/4})\\
							1\to2,\quad -1\to-2&\cosh(N^{-1/4})\\
							1\to-1,\quad -1\to1&a
						\end{array}
					\end{equation*}
					By considering the continuous-time simple random walk $\mathsf X_t$ with rates all equal to $\cosh(N^{-1/4})$, and using some coupling argument, we can show that 
					$$\mathbb{L}_{0}^{t}(D) \leq \frac{1}{2a}\mathbb{L}_{0}^{t}(\mathsf{X}),$$
					so that 
					\begin{equation}\label{local-time_D-X}
						\widehat{\mathbf{E}}_{0}\left[e^{C_{\alpha}N^{-1/2}\mathbb{L}_{0}^{tN}(D)}\right]
						\leq E^{\mathsf{X}}\left[e^{\frac{C_{\alpha}N^{-1/2}}{2a}\mathbb{L}_{0}^{tN}(\mathsf{X})}\right]\,.
					\end{equation}
					
					Therefore, everything boils down to estimating the local times for a continuous time symmetric random walk with transition rates $\cosh(N^{-1/4})$. 
					By standard application of Fourier transformation for the transition kernel of $\mathsf{X}_{t}$ one shows the first moment estimate
					\begin{equation*}
						E^{X}_{0}[\mathbb{L}_{0}^{tN}(\mathsf{X})]\leq \sqrt{\frac{\pi N t}{2}}\,.
					\end{equation*}
					Then, for $t\leq t^{*}$ with $t^{*}$ small enough, we may apply the Khas’minskii Lemma  \cite{Khasminskii1959}, obtaining that 
					\begin{equation*}
						E_{0}^{\mathsf{X}}\left[e^{\frac{C_{\alpha}N^{-1/2}}{2a}\mathbb{L}_{0}^{Nt}}\right]\leq \frac{1}{1-\frac{C_{\alpha}N^{-1/2}}{2a}E_{0}^{\mathsf{X}}[\mathbb{L}_{0}^{Nt}]}\leq  1+C\sqrt{t}\,.
					\end{equation*}
					Finally, by using the monotonicity in time of $\mathbb{L}_{0}^{t}(\mathsf{X})$ and conditional expectation we obtain the same bound for all $t\in [0,T]$. We conclude that 
					$$ \mathbb{E}\left[(Z_{t-s}^{(N)}-1)^{2}\right] \leq C (t-s)^{1/2}$$
					so that 
					$$\mathbb{E}\left[\left|\mathrm{E}^{\mathcal{P},\omega}\left[\phi\left(\frac{X_{N}(Ns)-sd_{N}}{\sqrt{N}}\right)\left(q(t,X_{N}(Nt))-q(s,X_{N}(Ns))\right)\right]\right|\right] \leq C  (t-s)^{1/4},$$
					which concludes the proof.}
\end{proof}
To prove tightness we apply Aldous criterion \cite{aldous, billingsley}. This consists of proving two facts.
\begin{enumerate}
	\item For every $t \in [0,T]$
	the sequence $\left\langle \mathfrak{F}_{N}\left(t,\cdot\right),\phi\right\rangle$ is tight.
	\item For all $\epsilon>0$
	\begin{equation*}
		\begin{split}
			\lim_{\delta \to 0} \limsup_{N \to \infty} \sup_{\substack{\tau\in \mathbb{T}_{T} \\ \theta \leq \delta}} 
			\mathbb{P} \left( |\left\langle \mathfrak{F}_{N}\left(\tau+\theta,\cdot\right),\phi\right\rangle{\color{black}-}\left\langle \mathfrak{F}_{N}\left(\tau,\cdot\right),\phi\right\rangle| > \epsilon \right)=0{\color{black}\,,}
		\end{split}
	\end{equation*}
	{\color{black} where $\mathbb{T}_{T}$ is a family of stopping times bounded by $T\in \mathbb{R}$. }
\end{enumerate} 
\paragraph{Verification of Aldous criterion. }
(1) is a consequence of the finite dimensional convergence proved below{\color{black};}. 
We show the convergence of the finite dimensional distribution of the field $\mathfrak{F}_{N}(t,x)$ to  $\mathcal{U}(t,x)$. This is a consequence of the convergence of the interpolated field $\bar{\mathfrak{F}}_{N}$ in $C((0,T], \mathcal S'(\R))$. 
Given the weak convergence of $\bar{\mathfrak{F}}_{N}$, it suffices to show that 
for any {\color{black}times $t_1, \dots, t_k\in [0,T]$},  $\phi\in C_{c}^{\infty}(\mathbb{R})$ and for all $\epsilon>0$ we have that {\color{black}
	\begin{equation}
		\lim_{N\to \infty}\mathbb{P}\left( \sum_{i=1}^k \left|\left\langle \mathfrak{F}_{N}\left(t_i,\cdot\right),\phi\right\rangle -\left\langle \bar{\mathfrak{F}}_{N}\left(t_i,\cdot\right),\phi\right\rangle \right|>\epsilon\right)=0\,.
		\label{eq:convergenceinproba}
\end{equation}}
By the definition of interpolated field $\bar{\mathfrak{F}}_{N}$ we have the following estimate
\begin{equation*}
	\begin{split}
		\left|\left\langle \mathfrak{F}_{N}\left(t,\cdot\right),\phi\right\rangle -\left\langle \bar{\mathfrak{F}}_{N}\left(t,\cdot\right),\phi\right\rangle \right|\leq& \left|\left\langle \mathfrak{F}_{N}\left(t,\cdot\right),\phi\right\rangle -\left\langle \bar{\mathfrak{F}}_{N}\left(\frac{k}{N},\cdot\right),\phi\right\rangle \right|
		\\&+
		\left|\left\langle \mathfrak{F}_{N}\left(t,\cdot\right),\phi\right\rangle -\left\langle \bar{\mathfrak{F}}_{N}\left(\frac{k+1}{N},\cdot\right),\phi\right\rangle \right|\,.
	\end{split}
\end{equation*}
Now, using the estimate \eqref{estimate} and by using the fact that $\vert t-\frac{k}{N}\vert, \vert t-\frac{k+1}{N} \vert\leq \frac{T}{N}$ we have that 
\begin{equation*}
	\mathbb{E}\left[ \left|\left\langle \mathfrak{F}_{N}\left(t,\cdot\right),\phi\right\rangle -\left\langle \bar{\mathfrak{F}}_{N}\left(\frac{k}{N},\cdot\right),\phi\right\rangle \right|\right]\leq o(1)\,, 
\end{equation*}
so that \eqref{eq:convergenceinproba} holds by Markov's inequality. 

{\color{black}To prove (2) we proceed as follows. Let $\tau$ be a stopping time bounded by $T$ and let $\theta\in [0, \delta]$. By the tower property of conditional expectation we have 
\begin{equation*}
	\mathbb{E}\left[|\langle \mathfrak{F}_{N}(\tau+\theta,\cdot),\phi\rangle-\langle \mathfrak{F}_{N}(\tau,\cdot),\phi\rangle|\right]=\mathbb{E}\left[\mathbb{E}\left[|\langle \mathfrak{F}_{N}(\tau+\theta,\cdot),\phi\rangle-\langle \mathfrak{F}_{N}(\tau,\cdot),\phi\rangle|\bigg|\sigma_{\tau}\right]\right]\,,
\end{equation*}
where $\sigma_{\tau}$ is the stopped sigma algebra. Conditionally on $\sigma_{\tau}$, the realized value of $\tau$ and the configuration $\eta(N\tau,x)$ are now fixed. The configuration at time $\tau$ reads
\begin{equation*}
	\eta(N\tau,\cdot)=\sum_{z}\eta(N\tau,z)\delta_{z}\,.
\end{equation*}
Then, by the linearity of the stochastic flow it is enough to consider
\begin{equation*}
	\begin{split}
		&\mathbb{E}_{\delta_{x}}\left[|\langle \mathfrak{F}_{N}(\tau+\theta,\cdot),\phi\rangle-\langle \mathfrak{F}_{N}(\tau,\cdot),\phi\rangle|\bigg|\sigma_{\tau}\right]\,,
	\end{split}
\end{equation*} 
namely with expectation of the RWRE started at position $x$ at time $\tau$. Recall the definition of $q(t,x)$ of equation \eqref{q-definition} and observe that, by the strong Markov property, 
\begin{equation*}
	q(y,\tau+\theta)=q(x,\tau)q(y-x,\theta),\qquad K_{N\tau,N(\tau+\theta)}(x,y)\overset{d}{=}K_{0,N\theta}(0,y-x)\,. 
\end{equation*}
Then, we write
\begin{multline*}
	\left|\sum_{y}q(y,\theta+\tau)K_{N\tau,N(\tau+\theta)}(x,y)\phi\left(\frac{y-d_{N}(\tau+\theta)}{\sqrt{N}}\right)-\phi\left(\frac{x-d_{N}\tau}{\sqrt{N}}\right)q(\tau,x)\right|
	\\ \overset{d}{=} 
	q(x,\tau)\left|\sum_{y}q(y-x,\theta)K_{0,N\theta}(0,y-x)\phi\left(\frac{x-d_{N}\tau}{\sqrt{N}}+\frac{y-x-d_{N}\theta}{\sqrt{N}}\right)-\phi\left(\frac{x-d_{N}\tau}{\sqrt{N}}\right)\right|\,.
\end{multline*}
By taking the expectation and exploiting Lemma 5.2, we have that 
\begin{equation*}
	\begin{split}
		\mathbb{E}_{\delta_{x}}\left[|\langle \mathfrak{F}_{N}(\tau+\theta,\cdot),\phi\rangle-\langle \mathfrak{F}_{N}(\tau,\cdot),\phi\rangle|\right]
		\leq q(\tau,x)C\theta^{1/4}\,.
	\end{split}
\end{equation*}
We now remove the conditioning. By linearity of the stochastic flow we have that
\begin{equation*}
	\begin{split}
		\mathbb{E}\left[\mathbb{E}\left[|\langle \mathfrak{F}_{N}(\tau+\theta,\cdot),\phi\rangle-\langle \mathfrak{F}_{N}(\tau,\cdot),\phi\rangle|\bigg|\sigma_{\tau}\right]\right]\leq& C\theta^{1/4}\mathbb{E}\left[\sum_{x}q(\tau,x)\eta(\tau N,x)\right]
		\\\leq& C\theta^{1/4}\,.
	\end{split}
\end{equation*}
Above we have applied the optional stopping theorem and the martingale property to the continuous time martingale
\begin{equation*}
	M_{N}(t)=\sum_{x}q(x,t)\eta(Nt,x)\,.
\end{equation*}
We conclude that, uniformly in time, 
\begin{equation*}
	\mathbb{E}\left[|\langle \mathfrak{F}_{N}(\tau+\theta,\cdot),\phi\rangle-\langle \mathfrak{F}_{N}(\tau,\cdot),\phi\rangle|\right]\leq C\theta^{1/4}\,,
\end{equation*} 
concluding the proof.}
\end{proof}
\section{Computation of noise variance}\label{Section-noiceVariance}
The goal of this section is to compute the value of the noise strength  $\gamma$ in Theorem \ref{Thm-convergence}. 
	\begin{proposition}\label{proposition-exact-constant}
For the stochastic flow defined in Section \ref{section-KMP-SFK}, 
	\begin{equation*}
		\gamma^{2}=\frac{1}{{\color{black}2}\alpha}\,.
	\end{equation*}
\end{proposition}
Proposition \ref{proposition-exact-constant} is proved in the rest of this section. This proof is based on an approximation argument. First we introduce a discrete-time counterpart of the RWRE, depending on a parameter $\epsilon$ (Definition \ref{definition-RWREdiscrete}), from which we define the discrete-time noise variance $\gamma_{\epsilon}^{2}$. Second, we show that the approximation is well-chosen in the sense that $\gamma_{\epsilon}^{2}$ converges to $\gamma^{2}$ as $\epsilon\to 0$ (Proposition \ref{proposition-limit}). Third, we compute the limit of $\gamma_{\epsilon}^{2}$ (Proposition \ref{proposition-limitingConstant}), which yields  the value of $\gamma^2$. 
\subsection{Definition of the discrete-time process}
\begin{definition} \label{definition-RWREdiscrete}
	On a probability space $(\Omega^{(\epsilon)}, \mathcal{F}^{(\epsilon)}, \mathbb{P}^{(\epsilon)})$, we define the following random variables: for each bond $(x,x+1)$, let  $\left(\mathcal{I}_{x,x+1}^{(n)}\right)_{n\in \mathbb{Z}_{\geq 0}}$ be a Bernoulli process with parameter $\epsilon\in [0,1]$,  independent for each bond.

	For each bond $(x,x+1)$ and each $n\in \mathbb{Z}_{\geq 0}$,  let  $B_{x,x+1}^{(n)}$ be an independent random variable distributed as  Beta$(\alpha,\alpha)$. 
	We denote a realization of the  environment by $(\text{Bin}, \omega)\in \Omega^{(\epsilon)}$, where  $\text{Bin}$ is a realization of the Bernoulli processes and $\omega$ is a realization of the Beta variables.

	Given $(\text{Bin}, \omega)\in \Omega^{(\epsilon)}$, we define a random walk $(X^{\epsilon}(n))_{n\in\mathbb{Z}_{\geq 0}}$ as follows. If at discrete-time $n$, $X^{\epsilon}(n)=x$,  the update rule is as follows:
	\begin{itemize}[leftmargin=12pt]
		\item If $\mathcal{I}^{(n)}_{x-1,x} = \mathcal{I}^{(n)}_{x,x+1} = 1$ or if $\mathcal{I}^{(n)}_{x-1,x} = \mathcal{I}^{(n)}_{x,x+1} = 0$ the particle does not move. These events happen with probability $\epsilon^{2}$ and $1-2\epsilon(1-\epsilon)$ respectively. 
		\item If $\mathcal{I}^{(n)}_{x-1,x} = 1$ and $\mathcal{I}^{(n)}_{x,x+1} = 0$: the walker moves to ${\color{black}x-1}$ with probability ${\color{black} B^{(n)}_{x-1,x}}$, or remains at $x$ with probability ${\color{black}1-B^{(n)}_{x-1,x}}$. This event happens with probability $\epsilon(1-\epsilon)$. 
		\item If $\mathcal{I}^{(n)}_{x-1,x} = 0$ and $\mathcal{I}^{n}_{x,x+1} = 1$: the walker moves to ${\color{black}x+1}$ with probability ${\color{black}1-B^{(n)}_{x,x+1}}$, or remains at  $x$ with probability ${\color{black}B^{(n)}_{x,x+1}}$. This event happens with probability $\epsilon(1-\epsilon)$.
	\end{itemize}
\end{definition}

The law of the discrete time RWRE of Definition \ref{definition-RWREdiscrete} can be described by the quenched transition kernel 
\begin{equation*}
	K_{n-1,n}^{\epsilon}(x,y):=\mathrm{P}^{Bin,\omega}(X^{\epsilon}(n)=y\,|\, X^{\epsilon}(n-1)=x)\,. 
\end{equation*}
Above we have denoted by $\mathrm{P}^{Bin,\omega}$ the quenched path space probability measure, by conditioning with respect to the Beta-variables and to the $\text{Binomial}(n,\epsilon)$ process, obtained as a sum of $n$ independent $\text{Bernoulli}(\epsilon)$ variables. Again, for all $N\in\mathbb{Z}_{\geq 0}$ we denote by  $K_{N}^{\epsilon}=K_{0,1}^{\epsilon}\cdots K_{N-1,N}^{\epsilon}$ the quenched $N$-steps transition kernel. 

We now introduce a transition kernel for the difference of two continuous time walkers, for which we derive the invariant measure. This will be useful to characterize the stationary distribution of its discrete-time counterpart. To this aim, consider two copies of the RWRE of Definition \ref{definition-RWRE} $(X_{1}(t))_{t\geq 0}$, $(X_{2}(t))_{t\geq 0}$. We denote the annealed transition kernel for the Markov process $\left(Z(t)\right)_{t\geq 0}:=\left(X_{1}(t)-X_{2}(t)\right)_{t\geq 0}$ by
\begin{equation*}
	\mathbf{p}_{\text{dif},t}(z,a):=\mathbb{E}\left[\sum_{y\in \mathbb{Z}}\mathbf{p}^{(2)}_{t}((z,0)\to (y_{2}+a,y_{2}))\right]\,.
\end{equation*}
In what follows, we drop the subscript $t$ when $t=1$ (c.f. with \eqref{p-dif}). We  report the following Lemma.
\begin{lemma}\label{Lemma-inv-pdiff}
	Let $\pi^{\text{inv}}_{\alpha}$ be the product measure on configurations $x\in \mathbb{Z}$ such that $\pi^{\text{inv}}_{\alpha}(x)=\mathbb{E}[\mu_{\alpha}(x)\mu_{\alpha}(0)]$. Then, the Markov process $(Z(t))_{t\geq 0}$, with annealed transition kernel $\mathbf{p}_{\text{dif},t}(\cdot,\cdot)$, has a unique invariant measure given by $ \pi^{\text{inv}}_{\alpha}$. 
\end{lemma}
Its proof is an immediate consequence of Lemma \ref{Lemma-invariantMeasure-K}, then is omitted.
For the rest of the paper, we consider the invariant measure $\pi^{\text{inv}}_{\alpha}$ to be re-normalized as
\begin{equation}\label{normalized-piInv}
	\pi^{\text{inv}}_{\alpha}(x)=\begin{cases}
		1&\text{if}\quad x\neq 0\\
		\frac{(\alpha+1)}{\alpha}&\text{if}\quad x=0
	\end{cases}\,.
\end{equation}
In analogy to \eqref{mu}, we introduce the annealed kernel for the one point as 
\begin{equation*}
	\mathbf{p}_{n}^{\epsilon}(x):=\mathbb{E}^{(\epsilon)}\left[K_{n}^{\epsilon}(0,x)\right]\,.
\end{equation*}
Moreover, we introduce the annealed kernel for the discrete-time two-points motion as
\begin{equation}\label{p-espsilon}
	\mathbf{p}_{n}^{(2),\epsilon}\left((x_{1},x_{2})\to (y_{1},y_{2})\right):=\mathbb{E}^{(\epsilon)}\left[K_{n}^{\epsilon}(x_{1},y_{1})K_{n}^{\epsilon}(x_{2},y_{2})\right]\,.
\end{equation}
Above, we have denoted by $E_{Bin,\omega}$ the expectation with respect to the environment and to the Binomial process. Furthermore, we define 
\begin{equation}\label{pdif-epsilon}
	\mathbf{p}_{\text{dif},n}^{\epsilon}(x,y):=\sum_{y_{1}\in\mathbb{Z}}{\color{black}\mathbf{p}^{(2),\epsilon}_{n}}\left((x,0)\to(y_{1}+y,y_{1})\right)\,.
\end{equation}
\begin{definition}\label{remark-Zn}
	We denote by $(Z^{\epsilon}(n))_{n\in \mathbb{Z}_{\geq 0}}$ the discrete-time Markov process with transition kernel $\mathbf{p}_{\text{dif},n}^{\epsilon}(\cdot,\cdot)$. 
\end{definition}
We observe that $\mathbf{p}_{\text{dif},n}^{\epsilon}(\cdot,\cdot)$ is invariant with respect to the measure $\pi^{\text{inv}}_{\alpha}$ of Lemma \ref{Lemma-inv-pdiff}. This is inherited from $\mathbf{p}_{\text{dif},t}(\cdot,\cdot)$. 
\subsection{Convergence of $\gamma_{\epsilon}$ to $\gamma$}
	We rewrite in a more convenient {\color{black}form} the noise variance $\gamma^{2}$, introduced in \eqref{constan-general}. Let 
\begin{align}
	N:=&\frac{1}{2}\sum_{z\in\mathbb{Z}}\left[\sum_{x,y\in \mathbb{Z}}(x-y)^{2}\mathbf{p}(x)\mathbf{p}(y)-\sum_{a\in \mathbb{Z}}(a-z)^{2}\mathbf{p}_{\text{dif}}(z,a)\right]\pi^{\text{inv}}_{\alpha}(z)\label{old-enne}\,,
	\\D:=&\sum_{z\in \mathbb{Z}}\left[\sum_{a\in \mathbb{Z}}{\color{black}\mathbf{p}_{\text{dif}}(z,a)}|a|-|z|\right]\pi^{\text{inv}}_{\alpha}(z)\,,\label{old-DI}
\end{align}
be the numerator and the denominator of the constant $\gamma^{2}$, respectively, so that  $\gamma^{2}=\frac{N}{D}$. 
	We introduce the discrete-time counterpart of the numerator as
	\begin{equation*}
		N^{\epsilon}:=\frac{1}{2}\sum_{z\in\mathbb{Z}}\left[\sum_{x,y\in \mathbb{Z}}(x-y)^{2}\mathbf{p}_{\lfloor \epsilon^{-1}\rfloor}^{\epsilon}(x)\mathbf{p}_{\lfloor \epsilon^{-1}\rfloor}^{\epsilon}(y)-\sum_{a\in \mathbb{Z}}(a-z)^{2}\mathbf{p}_{\text{dif},\lfloor \epsilon^{-1}\rfloor}^{\epsilon}(z,a)\right]\pi^{\text{inv}}_{\alpha}(z)\,.
	\end{equation*}
Similarly, we introduce the discrete-time counterpart of the denominator as 
\begin{equation*}
	D^{\epsilon}:=\sum_{z\in \mathbb{Z}}\pi^{\text{inv}}_{\alpha}(z)\left[\sum_{a\in \mathbb{Z}}\mathbf{p}_{\text{dif},\lfloor \epsilon^{-1}\rfloor}^{\epsilon}(z,a)|a|-|z|\right]\,.
\end{equation*}
Furthermore, we define the discrete-time constant as
\begin{equation}\label{epsilon-constant}
	\gamma_{\epsilon}^{2}:=\frac{N^{\epsilon}}{D^{\epsilon}}\,.
\end{equation}
\begin{proposition}\label{proposition-limit}
	We have the convergence
	\begin{equation*}
		\lim_{\epsilon\to 0}\gamma_{\epsilon}^{2}=\gamma^{2}\,.
	\end{equation*}
\end{proposition}
Proposition \ref{proposition-limit} comes from the fact that the discretization $X^{\epsilon}(n)$ converges to the RWRE  $X(t)$ from Definition \ref{definition-RWRE} in a suitable sense. The following Lemma will be useful {\color{black} to prove Proposition \ref{proposition-limit}}.
\begin{lemma}\label{lemma-convergence2PTS}
	For all $x_{1},x_{2},y_{1},y_{2}\in \mathbb{Z}$ we have that
	\begin{equation}\label{convergence-p2}
		\lim_{\epsilon\to 0} \mathbf{p}_{\lfloor\epsilon^{-1}\rfloor}^{(2),\epsilon}\left((x_{1},x_{2})\to (y_{1},y_{2})\right)=\mathbf{p}^{(2)}\left((x_{1},x_{2})\to (y_{1},y_{2})\right)\,.
	\end{equation}
	Moreover, for all $x\in \mathbb{Z}$, we have that 
	\begin{equation}\label{convergnce-mu}
		\lim_{\epsilon\to 0}\mathbf{p}_{\lfloor\epsilon^{-1}\rfloor}^{\epsilon}(x)=\mathbf{p}(x)\,.
	\end{equation}
\end{lemma}
\begin{proof}
	For $\epsilon\in [0,1]$, for all $x\in \mathbb{Z}$ and $t\in [0,T]$ we define 
	\begin{equation}\label{Q-process}
		\mathcal{Q}^{\epsilon}_{x,x+1}(t):=\sum_{i=1}^{\lfloor\epsilon^{-1}t\rfloor}\mathcal{I}_{x,x+1}^{i}\sim \text{Binomial}(\lfloor\epsilon^{-1}t\rfloor,\epsilon)\,.
	\end{equation}
	Then, $\epsilon\in[0,1]$ we introduce the counting process 
	\begin{equation*}
		\left( \mathbf{Q}^{\epsilon}(t)\right)_{t\in[0,T]}:=\left(\ldots,\mathcal{Q}^{\epsilon}_{x-1,x}(t), \mathcal{Q}^{\epsilon}_{x,x+1}(t),\mathcal{Q}^{\epsilon}_{x+1,x+2}(t)\,\ldots\right)_{t\in[0,T]}
	\end{equation*}
	Here, for each bond, $\mathcal{Q}_{x,x+1}^{\epsilon}(t)$ is an independent copy of \eqref{Q-process} in which we have the sum of $\mathcal{I}^{i}_{x,x+1}$ random variable. By using \cite[Example 12.3]{billingsley} and for all $(x,x+1)$ we have that 
	\begin{equation*}
		\left(	\mathcal{Q}^{\epsilon}_{x,x+1}(t)\right)_{t\in[0,T]}\xrightarrow[\epsilon\to 0]{d}(N(t))_{t\in[0,T]}
	\end{equation*}
	in the Skorokhod space $D([0,T],\mathbb{Z}_{\geq 0}^{\mathbb{Z}})$. Therefore, by using the definition of the kernel \eqref{p-espsilon} and the law of total expectation, we have that 
	\begin{equation*}
		\mathbf{p}_{\lfloor\epsilon^{-1}\rfloor}^{(2),\epsilon}\left((x_{1},x_{2})\to (y_{1},y_{2})\right)=\mathbb{E}^{(\epsilon)}\left[\mathbb{E}^{(\epsilon)}\left[K_{\lfloor \epsilon^{-1}\rfloor}^{\epsilon}(x_{1},y_{1})K_{\lfloor \epsilon^{-1}\rfloor}^{\epsilon}(x_{2},y_{2})|\bm{Q}^{\epsilon}(t)\right]\right]\,.
	\end{equation*}
	The mapping  \begin{equation*}
		\bm{Q}^{\epsilon}(t)\to \mathbb{E}^{(\epsilon)}\left[\mathbb{E}^{(\epsilon)}\left[K_{\lfloor \epsilon^{-1}\rfloor}^{\epsilon}(x_{1},y_{1})K_{\lfloor \epsilon^{-1}\rfloor}^{\epsilon}(x_{2},y_{2})|\bm{Q}^{\epsilon}(t)\right]\right]       
	\end{equation*} 
	can be seen as continuous and bounded function of the trajectories of the stochastic process $\bm{Q}^{\epsilon}(t)$, for $t\in [0,T]$. Therefore, by Portmanteau theorem, we have the result. With the same argument one proves \eqref{convergnce-mu}. 
\end{proof}
\begin{proof}[Proof of Proposition \ref{proposition-limit}] It suffices to show that $N^{\epsilon}$ and $D^{\epsilon}$ converge to $N$ and $D$ respectively as $\epsilon$ goes to zero. 
We may write the numerator $N^{\epsilon}$ as 
\begin{equation*}
	N^{\epsilon}=\sum_{z\in \mathbb{Z}}\pi^{\text{inv}}_{\alpha}(z)\left(\sum_{x,y\in \mathbb{Z}}xy\mathbf{p}^{(2),\epsilon}_{\lfloor \epsilon^{-1}\rfloor}((z,0)\to(x+z,y))-\sum_{x,y}xy\mathbf{p}_{\lfloor \epsilon^{-1}\rfloor}^{\epsilon}(x+z)\mathbf{p}_{\lfloor \epsilon^{-1}\rfloor}^{\epsilon}(y)\right)\,.
\end{equation*}
{\color{black}By a standard application of uniform integrability, we have that, for fixed $z\in \mathbb{Z}$, 
\begin{equation*}
	\sum_{x,y}xy\mathbf{p}_{\lfloor \epsilon^{-1}\rfloor}^{(2),\epsilon}((z,0)\to (x+z,y))-\sum_{x,y}xy \mathbf{p}_{\lfloor \epsilon^{-1}\rfloor}^{\epsilon}(x)\mathbf{p}_{\lfloor \epsilon^{-1}\rfloor}^{\epsilon}(y)
\end{equation*}
converges to 
\begin{equation*}
	\sum_{x,y}xy\mathbf{p}_{1}^{(2)}((z,0)\to (x+z,y))-\sum_{x,y}xy \mathbf{p}_{1}(x)\mathbf{p}_{1}(y)
\end{equation*}
Then,} by performing a similar argument to the one done in the proof of \textbf{HP 5)} in Proposition \ref{proposition-assumptions-verify}, we obtain that there exists an exponentially decaying  and $\ell^{1}(\mathbb{Z}_{\geq 0})$ function $G_{\text{Decay}}:\mathbb{Z}\to[0,\infty)$ such that, uniformly in $\epsilon$,  
\begin{equation*}
{\color{black}\bigg|}\sum_{x,y\in \mathbb{Z}}xy\mathbf{p}^{(2),\epsilon}_{\lfloor\epsilon^{-1}\rfloor}((z,0)\to(x+z,y))-\sum_{x,y}xy\mathbf{p}_{\lfloor\epsilon^{-1}\rfloor}^{\epsilon}(x)\mathbf{p}^{\epsilon}_{\lfloor\epsilon^{-1}\rfloor}(y){\color{black}\bigg|}\leq G_{\text{Decay}}(|z|)\,.
\end{equation*}
Therefore, by Lemma \ref{lemma-convergence2PTS} and the dominated convergence theorem, one obtains that $\lim_{\epsilon\to 0}N^{\epsilon}=N$. Likewise  one proves that $\lim_{\epsilon\to 0}D^{\epsilon}=D$. 
\end{proof}
\subsection{Computation of the limit}

\begin{proposition}\label{proposition-limitingConstant}
	Recall $\gamma_{\epsilon}$ from \eqref{epsilon-constant}. We have
	\begin{equation}\label{statement-proposition-gamma}
		\lim_{\epsilon\to 0}\gamma_{\epsilon}^{2}=\frac{1}{{\color{black}2}\alpha}\,.
	\end{equation}
\end{proposition}
\begin{proof} We treat the discrete-time numerator and denominator separately. 
	Let us start with the numerator. We re-write $N^{\epsilon}$ as
	\begin{equation}\label{N-epsilon}
		N^{\epsilon}=	\sum_{z\in \mathbb{Z}}\pi^{\text{inv}}_{\alpha}(z)\mathbb{COV}^{(\epsilon)}\left(\mathrm{E}_{z}^{(\epsilon)}\left[X^{\epsilon}(\lfloor \epsilon^{-1}\rfloor)-z\right],\mathrm{E}_{0}^{(\epsilon)}[Y^{\epsilon}(\lfloor \epsilon^{-1}\rfloor)]\right)\,.
	\end{equation}
		Above, we have denoted by $\mathrm{E}^{(\epsilon)}_{z}$ the quenched expectation with respect to the path space measure of the RWRE initialized at $z$. 
			Using the telescopic identity 
		\begin{equation}\label{telescopic}
			X^{\epsilon}(\lfloor\epsilon^{-1}\rfloor)-X^{\epsilon}(0)=\sum_{i=0}^{\lfloor\epsilon^{-1}\rfloor-1}\left(X^{\epsilon}(i+1)-X^{\epsilon}(i)\right)\,,
		\end{equation}
		and using the fact that the two walkers can be correlated only through the same Bernoulli event, we have that 
		\begin{equation*}
			\begin{split}
				N^{\epsilon}=&
				\sum_{i=0}^{\lfloor\epsilon^{-1}\rfloor-1}\sum_{z\in \mathbb{Z}}\pi^{\text{inv}}_{\alpha}(z)\mathbb{COV}^{(\epsilon)}\left(\mathrm{E}_{z}^{(\epsilon)}\left[X^{\epsilon}(i+1)-X^{\epsilon}( i)\right],\mathrm{E}_{0}^{(\epsilon)}\left[Y^{\epsilon}(i+1)-Y^{\epsilon}( i)\right]\right)\,.
			\end{split}
		\end{equation*}
		By using the invariance property of the measure $\pi^{\text{inv}}_{\alpha}$ we have that 
		\begin{equation*}
			\begin{split}
				N^{\epsilon}=
				\lfloor\epsilon^{-1}\rfloor\sum_{z\in \mathbb{Z}}\pi^{\text{inv}}_{\alpha}(z)\mathbb{COV}^{(\epsilon)}\left(\mathrm{E}_{z}^{(\epsilon)}\left[X^{\epsilon}(1)-z\right],\mathrm{E}_{0}^{(\epsilon)}\left[Y^{\epsilon}(1)\right]\right)\,.
			\end{split}
		\end{equation*}
		Therefore, it is enough to compute the covariance for one step of the chain, which we re-write as $\mathbb{COV}^{(\epsilon)}\left(\sum_{x\in \mathbb{Z}}(x-z)K_{1}^{\epsilon}(z,x),\sum_{y\in \mathbb{Z}}yK_{1}^{\epsilon}(0,y)\right)$. 

	If $|z|>1$, this covariance is trivially vanishing, then we consider the remaining cases. 
In what follows, for the sake of notation, we denote by $\bm{\mathcal{I}}$ a random vector containing the $\text{Bernoulli}(\epsilon)$ variables, associated with the environment and relevant for the considered case.

\item{\underline{Case $z=0$.}} Here, we have that 
\begin{equation*}
	\mathbb{COV}^{(\epsilon)}\left(\sum_{x\in \mathbb{Z}}x K_{1}^{\epsilon}(0,x),\sum_{y\in \mathbb{Z}}yK_{1}^{\epsilon}(0,y)\right)=\mathbb{VAR}^{(\epsilon)}\left(\sum_{x\in \mathbb{Z}}xK_{1}^{\epsilon}(0,x)\right)\,.
\end{equation*}
The relevant $\text{Bernoulli}(\epsilon)$ variables are listed in  $\bm{\mathcal{I}}=(\mathcal{I}_{-1,0},\mathcal{I}_{0,1})\in \{0,1\}^{2}$. 
By the law of total variance we have that
\begin{equation*}
	\begin{split}
		&\mathbb{VAR}^{(\epsilon)}\left(\sum_{x\in \mathbb{Z}}xK_{1}^{\epsilon}(0,x)\right)
		\\=&\mathbb{E}^{(\epsilon)}\left[\mathbb{VAR}^{(\epsilon)}\left(\sum_{x\in \mathbb{Z}}xK_{1}^{\epsilon}(0,x)\bigg|\bm{\mathcal{I}}\right)\right]+\mathbb{VAR}^{(\epsilon)}\left(\mathbb{E}^{(\epsilon)}\left[\sum_{x\in \mathbb{Z}}xK_{1}^{\epsilon}(0,x)\bigg|\bm{\mathcal{I}}\right]\right)\,.
	\end{split}
\end{equation*}

For the first addend we obtain that 
\begin{equation*}
	\begin{split}
		\mathbb{E}^{(\epsilon)}\left[\mathbb{VAR}^{(\epsilon)}\left(\sum_{x\in \mathbb{Z}}xK_{1}^{\epsilon}(0,x)\bigg| \bm{\mathcal{I}}\right)\right]&= 
		2\epsilon \mathbb{VAR}^{(\epsilon)}(B_{0})+o(\epsilon)
		=
		\frac{\epsilon}{2(2\alpha+1)}+o(\epsilon)\,.
	\end{split}
\end{equation*}
{\color{black} The second addend gives
\begin{equation}
	\begin{split}
		\mathbb{VAR}^{(\epsilon)}\left(\mathbb{E}^{(\epsilon)}\left[\sum_{x}xK_{1}^{\epsilon}(0,x)\bigg|\mathbb{\mathcal{I}}\right]\right)=& \mathbb{VAR}^{(\epsilon)}\left(\mathbb{E}^{(\epsilon)}\left[B_{0,1}\mathcal{I}_{0,1}-(1-B_{-1,0})\mathcal{I}_{-1,0}\bigg|\mathbf{\mathcal{I}}\right]\right)
		\\=&\frac{1}{4}\mathbb{VAR}^{(\epsilon)}(\mathcal{I}_{0,1}-\mathcal{I}_{-1,0})
		\\=&\frac{1}{4}2\epsilon(1-\epsilon)=\frac{\epsilon}{2}+o(\epsilon)\,.
	\end{split}
\end{equation}
It follows that, summing the 2 contributions, we have
\begin{equation}
	\begin{split}
		\mathbb{VAR}^{(\epsilon)}\left(\sum_{x}xK_{1}^{\epsilon}(0,x)\right)= \frac{\epsilon}{2(2\alpha+1)}+\frac{\epsilon}{2}+o(\epsilon)= \frac{\alpha+1}{2\alpha+1}\epsilon+o(\epsilon)\,.
	\end{split}
\end{equation}}
\item{\underline{Case $z=1$.}} The relevant $\text{Bernoulli}(\epsilon)$ variables are listed in $\bm{\mathcal{I}}:=(\mathcal{I}_{-1,0},\mathcal{I}_{0,1},\mathcal{I}_{1,2})$. By applying the law of total covariances, we have that 
\begin{multline*}
		\mathbb{COV}^{(\epsilon)}\left(\sum_{x\in \mathbb{Z}}(x-1)K_{1}^{\epsilon}(1,x),\sum_{y\in \mathbb{Z}}yK_{1}^{\epsilon}(0,y)\right)
		=\\
		\mathbb{E}^{(\epsilon)}\left[\mathbb{COV}^{(\epsilon)}\left(\sum_{x\in \mathbb{Z}}(x-1)K_{1}^{\epsilon}(1,x),\sum_{y\in \mathbb{Z}}yK_{1}^{\epsilon}(0,y)\bigg|\bm{\mathcal{I}}\right)\right]
		\\+\mathbb{COV}^{(\epsilon)}\left(\mathbb{E}^{(\epsilon)}\left[\sum_{x\in \mathbb{Z}}(x-1)K_{1}^{\epsilon}(1,x)\bigg|\bm{\mathcal{I}}\right],\mathbb{E}^{(\epsilon)}\left[\sum_{y\in \mathbb{Z}}yK_{1}^{\epsilon}(0,y)\bigg|\bm{\mathcal{I}}\right]\right)\,.
\end{multline*}
The first addend is of order $\epsilon$ only when $\bm{\mathcal{I}}=(0,1,0)$, otherwise it is equal to zero or of order $o(\epsilon)$. Then, we write
\begin{equation}
	\begin{split}
		\mathbb{E}^{(\epsilon)}\left[\mathbb{COV}^{(\epsilon)}\left(\sum_{x\in \mathbb{Z}}(x-1)K_{1}^{\epsilon}(1,x),\sum_{y\in \mathbb{Z}}yK_{1}^{\epsilon}(0,y)\,\bigg|\,\bm{\mathcal{I}}\right)\right]=&
		\mathbb{VAR}^{(\epsilon)}(B_{0})\epsilon+o(\epsilon)
		\\=&
		\frac{\epsilon}{4(2\alpha+1)}+o(\epsilon)\,.
	\end{split}
\end{equation}
{\color{black}The second addend gives 
	\begin{equation}
		\begin{split}
			&\mathbb{COV}^{(\epsilon)}\left(\mathbb{E}^{(\epsilon)}\left[\sum_{x}(x-1)K_{1}^{\epsilon}(1,x)\bigg|\mathbf{\mathcal{I}}\right],\mathbb{E}^{(\epsilon)}\left[\sum_{y}yK_{1}^{\epsilon}(0,y)\bigg|\mathbf{\mathcal{I}}\right]\right) 
			\\=& \frac{1}{4}\mathbb{COV}^{(\epsilon)}\left(\mathcal{I}_{1,2}-\mathcal{I}_{0,1},\mathcal{I}_{0,1}-\mathcal{I}_{-1,0}\right)
			\\=&
			\frac{1}{4}\mathbb{E}^{(\epsilon)}\left[(\mathcal{I}_{1,2}-\mathcal{I}_{0,1})(\mathcal{I}_{0,1}-\mathcal{I}_{-1,0})\right]
			\\=& 
			\frac{1}{4}\left(\mathbb{E}^{\epsilon}[\mathcal{I}_{1,2}\mathcal{I}_{0,1}]+\mathbb{E}^{\epsilon}[\mathcal{I}_{0,1}\mathcal{I}_{-1,0}]-\mathbb{E}^{\epsilon}[\mathcal{I}_{1,2}\mathcal{I}_{-1,0}]-\mathbb{E}^{\epsilon}[\mathcal{I}_{0,1}^{2}]\right)
			\\=&
			\frac{1}{4}(\epsilon^{2}-\epsilon)\,.
		\end{split}
	\end{equation}
	It follows that 
	\begin{equation}\label{cov-z1}
		\mathbb{COV}^{(\epsilon)}\left(\sum_{x}(x-1)K_{1}^{\epsilon}(1,x),\sum_{y}yK_{1}^{\epsilon}(0,y)\right)= \frac{\eps
		}{4(2\alpha+1)}-\frac{1}{4}\epsilon+o(\epsilon)= -\frac{\alpha \epsilon}{2(2\alpha+1)}+o(\epsilon)
	\end{equation}}
\item{\underline{Case $z=-1$.}} By symmetry, one obtains again \eqref{cov-z1}. 

Using equation \eqref{N-epsilon} for $N^{\epsilon}$ and the measure $\pi^{\text{inv}}_{\alpha}$, written in \eqref{normalized-piInv}, we have that 
\begin{equation}\label{numerator-approx-final}
	N^{\epsilon}=\lfloor\epsilon^{-1}\rfloor
{\color{black}	\left(\epsilon \frac{(\alpha+1)}{\alpha}\frac{\alpha+1}{2\alpha+1}-\epsilon\frac{\alpha}{(2\alpha+1)}\right)+o(1)}
	=\frac{1}{\alpha}+o(1)\,.
\end{equation}

\bigskip 
Let us now turn to the denominator. 
		By using the telescopic identity \eqref{telescopic}, the invariance of the measure $\pi^{\text{inv}}_{\alpha}$ and by an argument similar to that done for the numerator, we obtain 
	\begin{equation}\label{denominator-approximation}
		D^{\epsilon}=\lfloor\epsilon^{-1}\rfloor
		\sum_{z\in \mathbb{Z}}\pi^{\text{inv}}_{\alpha}(z)\mathbb{E}^{(\epsilon)}\left[\mathrm{E}_{z}^{\text{dif},(\epsilon)}\left[|Z^{\epsilon}(1)|-|Z^{\epsilon}(0)|\right]\right]\,.
	\end{equation}
	Here, $\mathrm{E}_{z}^{\text{dif},(\epsilon)}$ denotes the quenched expectation with respect to the law of the Markov chain $(Z(n))_{n\in \mathbb{Z}_{\geq 0}}$ introduced in Definition \ref{remark-Zn}. Therefore, to find the exact expression for the denominator it is enough to compute the annealed expectation
	\begin{equation*}
		\begin{split}
			\mathbb{E}^{(\epsilon)}\left[\mathrm{E}_{z}^{\text{dif},(\epsilon)}\left[|Z^{\epsilon}(1)|\right]\right]-|z|=&\mathbb{E}^{(\epsilon)}\left[\mathbb{E}^{(\epsilon)}\left[\sum_{y\in \mathbb{Z}}K_{1}^{\epsilon}(0,y)\sum_{a\in \mathbb{Z}}K_{1}^{\epsilon}(z,a+y)|a|
			\bigg| \bm{\mathcal{I}}\right]\right]-|z|\,.
		\end{split}
	\end{equation*}
	For $|z|>1$, the denominator is vanishing. We then consider the remaining cases.
\item{\underline{Case $z=0$.}}  Here, the relevant $\text{Bernoulli}(\epsilon)$ variables are listed in $\bm{\mathcal{I}}=(\mathcal{I}_{-1,0},\mathcal{I}_{0,1})$. Then, we write
	\begin{equation*}
		\begin{split}
			\mathbb{E}^{(\epsilon)}\left[\mathbb{E}^{(\epsilon)}\left[\sum_{y\in \mathbb{Z}}K_{1}^{\epsilon}(0,y)\sum_{a\in \mathbb{Z}}K_{1}^{\epsilon}(0,a+y)|a|\bigg| \mathcal{I}\right]\right]=&
			\left(1-4\mathbb{VAR}^{(\epsilon)}(B)\right)\epsilon+o(\epsilon)
			\\=&
			\frac{2\alpha}{2\alpha+1}\epsilon+o(\epsilon)\,.
		\end{split}
	\end{equation*}
	
	\item{\underline{Case $z=1$.}} Here, the relevant $\text{Bernoulli}(\epsilon)$ variables are listed in  $\bm{\mathcal{I}}=(\mathcal{I}_{-1,0},\mathcal{I}_{0,1},\mathcal{I}_{1,2})$. 	 In this case, we have contribution of order at most $\epsilon$ only when $\mathcal{I}_{-1,0}=\mathcal{I}_{0,1}=\mathcal{I}_{1,2}=0$ (that is an event with probability $1-3\epsilon+o(\epsilon)$) or when  $\mathcal{I}_{-1,0}+\mathcal{I}_{0,1}+\mathcal{I}_{1,2}=1$ (that are events with probability $\epsilon+o(\epsilon)$). Then, we obtain
	\begin{equation}\label{z1-denominator}
		\begin{split}
			&\mathbb{E}^{(\epsilon)}\left[\mathbb{E}^{(\epsilon)}\left[\sum_{y\in \mathbb{Z}}K_{1}^{\epsilon}(0,y)\sum_{a\in \mathbb{Z}}K_{1}^{\epsilon}(1,a+y)|a|\bigg| \bm{\mathcal{I}}\right]\right]-1
\\=&
			  \left(\left(1+\mathbb{E}^{(\epsilon)}[B_{-1}]\right)+2\left(\frac{1}{4}-\mathbb{VAR}^{(\epsilon)}(B_{0})\right)
			  +\mathbb{E}^{(\epsilon)}[2(1-B_{1})+B_{1}]\right)\epsilon-3\epsilon+o(\epsilon)
\\=&
		\frac{\alpha}{2\alpha+1}\epsilon+o(\epsilon)
		\end{split}
	\end{equation}
\item{\underline{Case $z=-1$.}}  By symmetry, we have again \eqref{z1-denominator}.

	By computing the expectation with respect to the invariant measure $\pi^{\text{inv}}_{\alpha}$ of equation \eqref{normalized-piInv}, we obtain 
	\begin{equation}\label{denominator-approx-final}
		\begin{split}
			D^{\epsilon}=&\lfloor\epsilon^{-1}\rfloor\left(\frac{2\alpha}{2\alpha+1}\epsilon+\frac{2\alpha}{2\alpha+1}\frac{\alpha+1}{\alpha}\epsilon\right)+o(1)=2+o(1)\,.
		\end{split}
	\end{equation}
	Finally, by using \eqref{numerator-approx-final} and \eqref{denominator-approx-final}, we have that
	\begin{equation*}
		\lim_{\epsilon\to 0}\gamma_{\epsilon}^{2}=  \lim_{\epsilon\to 0}\frac{N^{\epsilon}}{D^{\epsilon}}=\frac{1}{{\color{black}2}\alpha}\,.
	\end{equation*}
\end{proof}
\appendix
\section{Discrete-time analogues}\label{appendix-DtModels}
In this section, we consider a discrete-time random walk in a i.i.d. beta-distributed random environment introduced in \cite{barraquand2015}. We explain its connection to  a discrete time energy redistribution model in a so-called ``brick-wall'' geometry. Even if the redistribution rule differs from the one of the KMP, this connection has been an important source of inspiration for this work. Finally, we explain that both models may also be seen as a Haar unitary random quantum circuit, or equivalently,  a toy model of directed waves in random media introduced in \cite{kardar1992}. 
\subsection{The discrete-time Beta random walk in random environment} \label{section-dt-RWRE}
Let $(B_{n,x})_{n,x\in\mathbb{Z}}$ be a collection of i.i.d. $\text{Beta}(\alpha,\alpha)$ distributed random variables, playing the role of the random environment. We denote by $(X(n))_{n\in \mathbb{Z}}\subset \mathbb{Z}$ a discrete-time random walk, moving in this random environment (Beta-RWRE). As above, $\mathrm{P}$ is the quenched path space measure, while $\mathbb{P}$ is the probability measure over the environment. The transition probabilities for the walker read 
\begin{equation*}
    \text{P}(X(n+1)=x+1|X(n)=x)=B_{n,x}\quad \text{and}\quad \text{P}(X(n+1)=x-1|X(n)=x)=1-B_{n,x}\,.
\end{equation*}
 Let $\mathsf{p}_n(x):=\mathrm P_0\left(X(n)=x\right)$ be the quenched $n$-steps transition probability for the walker, starting from $0$ and reaching $x$, for which we adopt the convention that it is non-zero only if $n$ and $x$ have the same parity. The following recursion relation holds
 	\begin{equation}\label{total-prob}
	\mathsf{p}_n(x)	=\accentset{\uparrow}{\mathsf{p}}_{n}(x)+\accentset{\downarrow}{\mathsf{p}}_{n}(x)=
		\mathsf{p}_{n-1}(x-1)B_{n-1,x-1}+\mathsf{p}_{n-1}(x+1)\left(1-B_{n-1,x+1}\right)\,,
	\end{equation}
    where we denoted
    \begin{equation*}
		\accentset{\uparrow}{\mathsf{p}}_{n}(x)=
		\text{P}\left(X(n)=x,\,X(n-1)=x-1\right),\quad 
		\accentset{\downarrow}{\mathsf{p}}_{n}(x)=	\text{P}\left(X(n)=x,\,X(n-1)=x+1\right)\,.
	\end{equation*}
We write equation \eqref{total-prob} in order to compare the model with the energy exchange model and the random quantum circuit discussed below.  This recursion appeared in the context of directed waves in random media \cite{kardar1992}, see also  \cite{PhysRevB.50.5119,kardar1994directed}. It was used  to compute $\mathbb{E}[\mathrm{E}[X(n)]^k]$ for $k=2$ in in \cite{kardar1992} and $k=3,4$ in  \cite{PhysRevB.50.5119}. 
However, this recursion is not particularly convenient to study the model. It turns out to be much better to consider, for any set $A$, 
$$q_n(x) = \mathrm{P}(X(0)\in A \vert X(-n)=x).$$
For a fixed time $n$ and $A={0}$, $(q_n(x))_{x\in \mathbb Z}$ and $(p_n(-x))_{x\in \mathbb Z}$ have the same distribution. It was shown in \cite{barraquand2015} that the evolution of the function 
$$ u(n,x_1, \dots, x_k) := \mathbb{E}\left[ \prod_{i=1}^k q_n(x_i) \right] $$
is solvable by Bethe ansatz (and the solution admits a $k$-fold contour integral expression). This was used in \cite{barraquand2015} to prove that $ \log \mathsf p_n(cn)$  has $n^{1/3}$ scale Tracy-Widom fluctuations, putting the model in the context of the  KPZ universality class.

Before discussing other related models, let us derive the stationary measure for the Beta-RWRE on a finite segment with reflecting boundaries. We recall that Beta RWRE on $\mathbb{Z}_{\geq 0}$ with a boundary condition was studied in \cite{barraquand2022}. Here we consider segment of length $N+1$, where the sites are denoted by $x\in \{0,\ldots,N\}$. 
   \begin{equation}\label{inhomogeneous-environment}
	B_{n,x}\sim \text{Beta}(\alpha_{x+1},\alpha_{x})\,\qquad \forall x\in \{0,\ldots, N-1\}\,,
\end{equation} 
For all $n\in\mathbb{Z}_{\geq 0}$ and for $x\in\{1,\ldots,N-1\} $ the recurrence is given by \eqref{total-prob}, while on the boundaries we have that 
    \begin{equation*}
		\mathsf{p}_{n}(0)=B_{n-1,1}\,\mathsf{p}_{n-1}(1),\quad 
		\mathsf{p}_{n}(N)=B_{n-1,N-1}\,\mathsf{p}_{n-1}(N-1)\,.
	\end{equation*} 
	Therefore, we say that the boundaries are reflecting.  
	\begin{lemma}\label{Lemma-stat-meas-RWRE-homo}
		The stationary measure of the Beta-RWRE on the finite segment $\{0,1,\ldots,N\}$ is given by 
		\begin{equation}\label{stationary-RWRE-non-Homo}
			\nu_{\text{stat}}=\bigotimes_{x=0}^{N}\nu_{\text{stat}}(x),\quad \text{where}\quad \nu_{\text{stat}}(x) \sim \mathrm{Gamma}(\alpha_{x}+\alpha_{x+1})\,.
		\end{equation} 
    \end{lemma}
\begin{proof} Stationarity follows from the elementary property: given the random variables
		\begin{equation*}
			\nu_{1},\nu_{2}\sim \text{Gamma}(a+b),\; \text{independent};\quad B_{1},B_{2}\sim \text{Beta}(a,b),\; \text{independent}\,,
		\end{equation*}
		then
		\begin{equation*}
			\nu_{1}B_{1}\sim \text{Gamma}(a),\;\nu_{2}(1-B_{2})\sim \text{Gamma}(b)\quad  \text{independent}\,.
		\end{equation*}  
Independence follows {\color{black}from} another elementary property: given two independent random variables $\nu\sim \text{Gamma}(a+b)$ and  $B\sim \text{Beta}(a,b)$, then  
\begin{equation*}
	(\nu B,\nu(1-B))\sim (\text{Gamma}(a),\text{Gamma}(b))\;\text{independent}\,.
\end{equation*}
\end{proof}
\subsection{Brick-wall KMP model}\label{section-BW-def}
We introduce the discrete-time Markov process $(\eta(n,x)_{x\in\mathbb{Z}})_{n\in \mathbb{Z}}$, which we call brick-wall KMP model  (BW-KMP). Any configuration at time $n\in\mathbb{Z}$ belongs to the state space $[0,1]^{\mathbb{Z}}$, where $\eta(n,x)$ represents the energy at site $x$. To describe its time evolution, we consider a sequence of i.i.d. $(B_{n,x})_{n,x\in \mathbb{Z}}$, such that $B_{n,x}\sim \text{Beta}(\alpha,\alpha)$ and we define the following updating rule: 
When $n$ is even, we redistribute the energy on bonds $(x,x+1)$ with $x$ even according to the rule 
		\begin{equation}\label{updatingRule-BW}
	\left(\eta(n+1,x),\eta(n+1,x+1)\right)=\left(1-B_{n,x},B_{n,x}\right)\left(\eta(n,x)+\eta(n,x+1)\right).
\end{equation}
When $n$ is odd, we do the same with $x$ odd.
There is a connection between the BW-KMP process and the Beta-RWRE defined in Section \ref{section-dt-RWRE}. 
\begin{proposition}\label{proposition-connetion-BW-RWRE}
	Assume that, at discrete-time $n_{0}\in \mathbb{Z}$, we have $\eta(n_{0},x)+\eta(n_{0},x+1)=\mathsf{p}_{n_{0}}(x)$ for all $x\in\mathbb{Z}$. 
		Then, for all $n\in \mathbb{Z}$ such that $n\geq n_{0}$, we have 
		\begin{equation}\label{connection-BW-RW}
        \begin{split}
			&\eta(n,x)
			=\mathsf{p}_{n-1}(x-1)B_{n-1,x-1}\quad \text{and}\quad 
		\eta(n,x+1)
		=\mathsf{p}_{n-1}(x+1)\left(1-B_{n-1,x+1}\right)\,.
        \end{split}
		\end{equation}
	
	\end{proposition}
	 \begin{proof} The proof is done by induction. By using \eqref{updatingRule-BW} we have that 
	 	\begin{equation*}
	 		\begin{split}
	 			\left(\eta(n_{0}+1,x),\eta(n_{0}+1,x+1)\right)=&\left(1-B_{n_{0},x},B_{n_{0},x}\right)(\eta(n_{0},x)+\eta(n_{0},x+1))\,.
	 		\end{split}
	 	\end{equation*}
	 	From the assumption, it follows that $\eta(n_{0}+1,x+1)=B_{n_{0},x}\mathsf{p}_{n_{0}}(x)$. We now repeat the argument for $n_{0}+1$ and $x+2$, obtaining 
	 		$\eta(n_{0}+1,x+2)=(1-B_{n_{0},x+2})\mathsf{p}_{n_{0}}(x+2)$. Now we assume that, for arbitrary $n> n_{0}$, we have  $\mathsf{p}_{n}(x)=\eta(n,x)+\eta(n,x+1)$ for all $x\in \mathbb{Z}$. Then, the same argument leads to
	 		\begin{equation*}
\eta(n+1,x+1)=B_{n,x}\mathsf{p}_{n}(x)\quad \text{and}\quad \eta(n+1,x+2)=(1-B_{n,x+2})\mathsf{p}_{n}(x+2)\,.
	 		\end{equation*}
	\end{proof}
\begin{remark}
Consider the BW-KMP process on the finite segment $\{1, \ldots, 2N\}$ with reflecting boundary conditions. Specifically, examine the process $\left(\eta(n, x)_{x \in \{1, \ldots, 2N\}}\right)_{ n \in \mathbb{Z}{\geq 0}}$ according to the rule  \eqref{updatingRule-BW} with $x,x+1 \in \{1, \ldots, 2N\}$,  
 see Figure \ref{fig:BW} for a pictorial representation, and choose the  environment  as \eqref{inhomogeneous-environment}. As an immediate consequence of Proposition \ref{proposition-connetion-BW-RWRE}, this BW-KMP on a segment is stationary with respect to the measure $\nu_{\text{stat}}$ in \eqref{stationary-RWRE-non-Homo}.
\end{remark}
\begin{figure}
	\begin{center}
		\definecolor{brickred}{rgb}{0.8, 0.25, 0.33}
		\begin{tikzpicture}[scale=1.1]
			\foreach \x in {1,2,...,8}
			\draw[thick] (\x,1) node[anchor=north]{\footnotesize $\eta(1,\x)$} -- (\x,6) node[anchor=south]{\footnotesize $\eta(6,\x)$};	
			\foreach \x in{1,3,5,7}
			{
			\foreach \y in {1,3,5}
			{
			\fill[brickred, draw=black, line width=2pt] (\x-0.4,\y+0.1) -- (\x+1.4,\y+0.1) -- (\x+1.4,\y+0.9) -- (\x-0.4,\y+0.9) --cycle;
			\draw (\x+0.5,\y+0.5) node{$B_{\y,\x}$};
			}
			}
			\foreach \x in{2,4,6}
			{
				\foreach \y in {2,4}
				{
					\fill[brickred, draw=black, line width=2pt] (\x-0.4,\y+0.1) -- (\x+1.4,\y+0.1) -- (\x+1.4,\y+0.9) -- (\x-0.4,\y+0.9) --cycle;
					\draw (\x+0.5,\y+0.5) node{$B_{\y,\x}$};
				}
			}	
			\draw (0.5,0) -- (8.5,0) node[anchor=north]{$x$};
			\foreach \x in {1,2,...,8}
			\draw (\x,-0.1) node[anchor=north]{\footnotesize $\x$} -- (\x,0.1);
			\draw[-stealth] (0,0.5) -- (0,6.5) node[anchor=east]{$n$};
			\foreach \y in {1,2,...,6}
			\draw (-0.1,\y) -- (0.1,\y);
		\end{tikzpicture}
		
	\end{center}
	\caption{Brick-wall KMP model on a segment of length $8$ between times $n=1$ and $n=6$. The vertical lines represent the energy $\eta(n,x)$ flowing upward in time (for simplicity of the picture, we only report the initial and the final energies). The boxes represent the redistribution mechanism, depending on a Beta variable, between sites $x$ and $x+1$ at time $n$ and according to the recursion relation \eqref{updatingRule-BW}. The reflecting boundary condition for the Beta RWRE translates in a very simple manner in the brick-wall description.}
	\label{fig:BW}
\end{figure}
\subsection{A model for directed waves in random media and random walk in random environment}
Motivated by a problem of directed waves in random media, \cite{kardar1992} considered a unitary variant of the stochastic heat equation with multiplicative noise (unitary means that it preserves the $L^2$ norm). The simplest model is the stochastic PDE 
$$ \partial_t \psi(t,x) = \frac{\I}{2} \partial_{xx} \psi(t,x) + \I \Xi(t,x) \psi(t,x),$$
but it is not well defined a priori, which was leading to some conflicting results in the literature -- see the discussion in \cite{kardar1992}. Saul, Kardar and Read argued that the directed polymer interpretation of the usual stochastic heat equation is key, and that unitary analogues should likewise  be defined through a path integral. These considerations led them to the following discrete model. We consider a wave function $\Psi_n(x)\in \C^2$ evolving at discrete times as follows. It is convenient to write 
$$ \Psi_n(x) = \begin{pmatrix}
	\Psi_n^+(x) \\ \Psi_n^-(x)
\end{pmatrix} $$ 
and imagine that the components $\Psi_n^{\pm}(x)$ are attached to the edges adjacent to $(n,x)$ in the lattice, $\Psi_n^+(x)$ is the component of wave function coming from above, i.e. coming to $(n-1,x+1)$, so it should be associated with the edge from  $(n-1,x+1)$ to $(n,x)$, and similarly for $\Psi_n^-(x)$.  

The time evolution is the following. Fix some $(n,x)$. It is convenient to write 
\begin{equation*}
	\mathbf{\Psi}_{n}^{\text{in}}(x)= \begin{pmatrix}
		\Psi_{n}^{\text{in},+}(x)\\ \Psi_{\text{in}}^{\text{in},-}(x)
	\end{pmatrix}:=\begin{pmatrix}
		\Psi^{+}_n(x)\\
		\Psi^{-}_n(x)
	\end{pmatrix}, \quad 	\mathbf{\Psi}_{n}^{\text{out}}(x)= \begin{pmatrix}
		\Psi_{n}^{\text{out},+}(x) \\ \Psi_{n}^{\text{out},-}(x)
	\end{pmatrix}:=\begin{pmatrix}
		\Psi^{+}_{n+1}(x-1)\\
		\Psi^{-}_{n+1}(x+1)
	\end{pmatrix}\,.
\end{equation*}
Then, the time evolution of the wave function is determined by 
\begin{equation*}
	\mathbf{\Psi}^{\text{out}}_n(x)=U_{n,x}\,\mathbf{\Psi}^{\text{in}}_n(x)\,.
\end{equation*} 
where $U_{n,x}$ is a Haar distributed $2\times 2$ unitary matrix (and the family  $U_{n,x}$ is i.i.d.). Since $U_{n,x}$ is unitary 
\begin{equation*}
	\Vert\mathbf{\Psi}^{\text{out}}\Vert^{2}=\Vert\mathbf{\Psi}^{\text{in}}\Vert^{2}\,.
\end{equation*} 
Moreover, as argued in \cite{kardar1992}, 
\begin{equation} \label{environment-kardar}
	B_{n,x} := \frac{\vert \Psi_{n}^{\text{out},-}(x)\vert^2}{\Vert \Psi_{n }^{\text{out}}(x)\Vert^2} \end{equation}
does not depend on $\mathbf{\Psi}^{\text{in}}$ (because the Haar measure is translation invariant) and is uniformly distributed (because columns of Haar unitary matrices are uniformly distributed on the complex sphere). 

\begin{remark}\label{beta-waveModel}
	One could consider a more general model where $\Psi_n^+(x)\in \C^N$,  $\Psi_n^+(x)\in \C^M$ and $U_{n,x}$ is Haar distributed in $U(N+M)$. The model would be very similar with $B_{n,x}\sim \mathrm{Beta}(N,M)$. 
\end{remark}
\begin{proposition}
	Given that at time $n_{0}\in \mathbb{Z}$  we have $\lVert \Psi_{n_{0}}^{\text{out}}(x)\Vert^2 =\mathsf{p}_{n_{0}}(x)$. 
	Then, for all $n\in \mathbb{Z}$ such that $n\geq n_{0}$ we have 
	\begin{equation*}
		\mathsf{p}_n(x)=\Vert \Psi_{n}(x)\Vert^2\,.
	\end{equation*}
\end{proposition}
\begin{proof}
	Using the environment \eqref{environment-kardar}, the proof is analogous to the one of Proposition \ref{proposition-connetion-BW-RWRE}.
\end{proof}
\begin{remark}
	We observe that the BW-KMP and the model described in this last subsection can be viewed as random quantum circuits, since the updating rule can be rephrased in a unitary dynamic evolving quantum states -- see for instance \cite{nahum2018operator,fisher2023random}.
\end{remark}
{\color{black} 
\section{Markov duality}
\label{sec:proofduality}
Recall the definitions of the KMP model (Definition \ref{definition-KMP}), the dual-KMP model (Definition \ref{def:dualKMP}) which is the $k$-point motion associated to the randomwalk in random environment in Definition \ref{definition-RWRE}.  Recall the expression of the duality functional 
\begin{equation*}
	D(\eta,\xi):=\prod_{x\in \mathbb{Z}}d(\eta(x),\xi(x)), \hspace{10pt}
	d(\eta(x),\xi(x))=\eta(x)^{\xi(x)}\frac{\Gamma(\alpha)}{\Gamma(\alpha+\xi(x))}\,.
\end{equation*}
We have the Markov duality 
\begin{equation}
	\mathbb E_{\eta(0,\cdot)}\left[ D(\eta(t,\cdot),\xi(0,\cdot))\right]=\mathbb E_{\xi(0,\cdot)}\left[ D(\eta(0,\cdot),\xi(t,\cdot))\right]\,,
\end{equation}
where the expectations on the right-hand-side and on the left-hand-side are taken with respect to the law of the KMP with initial condition $\eta(0,\cdot)$ and of the dual KMP with initial condition $\xi(0,\cdot)$, respectively. It was first proved in \cite{kipnis1982heat}  in the case $\alpha=1$ and in general in \cite{modelsOfTransport}. We present below an alternative proof based on the reversibility of the  stationary measure for the $k$-point motion of random walks in random environment. 
\begin{proof}
	Consider two configuration vectors $\xi=(\xi_{x})_{x\in\mathbb{Z}}$ and $(\hat{\xi}_{x})_{x\in \mathbb{Z}}$, where $\xi_{x}$ ($\hat{\xi}_{x}$) denotes the number of particles at site $x$. We assume that both vectors contains $k$ particles and we denote by $x_{1},\ldots, x_{k}$ and by $\hat{x}_{1},\ldots,\hat{x}_{k}$ their location. It is convenient to use the notation    $\mathrm{P}_{t}(\xi,\hat{\xi})=\mathbf{p}_{t}^{(k)}((x_{1},\ldots,x_{k})\to (\hat{x}_{1},\ldots,\hat{x}_{k}))$. By Lemma \ref{Lemma-invp2}, we have the detailed balance condition
	\begin{equation*}
		\nu^{\text{inv}}_{\alpha}(\xi)\mathrm{P}_{t}(\xi,\hat{\xi})=\nu^{\text{inv}}_{\alpha}(\hat{\xi})\mathrm{P}_{t}(\hat{\xi},\xi)\,.
	\end{equation*}
For fixed $(\eta(0,x))_{x\in \mathbb{Z}}$, we denote $f(\xi)=\prod_{x\in\mathbb{Z}}\eta(0,x)^{\xi(x)}$ and we write 
	\begin{equation*}
		\sum_{\hat{\xi}}\mathrm{P}_{t}(\xi,\hat{\xi})\frac{f(\hat{\xi})}{\nu^{\text{inv}}_{\alpha}(\hat{\xi})}=\frac{1}{\nu^{\text{inv}}_{\alpha}(\xi)}\sum_{\hat{\xi}}\mathrm{P}_{t}(\hat{\xi},\xi)f(\hat{\xi})\,.
	\end{equation*}
	We analyze separately the two sides of the equation above. On the left-hand-side, we have that 
	\begin{equation*}
		\sum_{\hat{\xi}}\mathrm{P}_{t}(\xi,\hat{\xi})\frac{f(\hat{\xi})}{\nu^{\text{inv}}_{\alpha}(\hat{\xi})}=\mathbb{E}_{\xi}\left[\prod_{x\in \mathbb{Z}}\eta(0,x)^{\xi(t,x)}\frac{\Gamma(\alpha)}{\Gamma(\alpha+\xi(t,x))}\right]
	\end{equation*}
	On the right-hand-side we have that the quantity $\nu^{\text{inv}}_{\alpha}(\xi):=\prod_{x\in \mathbb{Z}}\frac{\Gamma(\alpha+\xi(x))}{\Gamma(\alpha)}$, while we have the following chain of equalities: 
	\begin{equation*}
		\begin{split}
			\sum_{\hat{\xi}}\mathrm{P}_{t}(\hat{\xi},\xi)f(\hat{\xi})=&\sum_{\hat{x}_{1},\ldots,\hat{x}_{k}\in \mathbb{Z}}\prod_{x\in \mathbb{Z}}\eta(0,x)^{\hat{\xi}(x)}\mathbb{E}\left[K_{t}(\hat{x}_{1},x_{1})\cdots K_{t}(\hat{x}_{k},x_{k})\right]
			\\=&
			\mathbb{E}\left[\sum_{\hat{x}_{1}\in \mathbb{Z}}\cdots \sum_{\hat{x}_{k}\in \mathbb{Z}}\prod_{\ell=1}^{k}\eta(0,\hat{x}_{\ell})K_{t}(\hat{x}_{\ell},x_{\ell})\right]
			\\=&
			\mathbb{E}\left[\prod_{x\in \mathbb{Z}}\left(\sum_{y\in \mathbb{Z}}\eta(0,y)K_{t}(y,x)\right)^{\xi(x)}\right]
			\\=&
			\mathbb{E}_{\eta(0)}\left[\prod_{x\in\mathbb{Z}}\eta(t,x)^{\xi(x)}\right]
		\end{split}
	\end{equation*}
	Therefore, we immediately read off the duality relation by setting $\xi(0,x)=\xi(x)$ for all $x$. 
\end{proof}
\section{Details of the proof of Lemma \ref{Lemma5-2}}\label{Appendix-details5-2}
\subsection{Details to obtain \eqref{local-time_D-X}} 
We define
\begin{equation*}
	(L^{(2)}f)(\bm{x})=\sum_{\bm{y}\neq \bm{x}}q^{(2)}(\bm{x},\bm{y})[f(\bm{y})-f(\bm{x})]
\end{equation*}
the Markov generator of the $2-$points motion and let
\begin{equation*}
	(L^{\text{ind}}f)(\bm{x})=\sum_{\bm{y}\neq \bm{x}}q^{\text{ind}}(\bm{x},\bm{y})[f(\bm{y})-f(\bm{x})]
\end{equation*}
be the Markov generator two independent copies of the annealed one-point motion. It is straightforward computation to show that 
\begin{equation*}
	L^{\text{ind}}h_{N}(\bm{x})=c_{N}h_{N}(\bm{x})
\end{equation*}
The function $u(r,\bm{x})$ obeys to
\begin{equation*}
	\begin{cases}\partial_{r}u(r,\bm{x})= (L^{(2)}-c_{N})u(r,\bm{x})
		\\u(0,\bm{x})=h_{N}(\bm{x})
	\end{cases}
\end{equation*}
Then, we further introduce the function
\begin{equation*}
	v(r,\bm{x}):= h_{N}^{-1}(\bm{x})\mathbf{E}\left[h_{N}(\bm{X}(r))e^{-c_{N}r}\right]=h_{N}^{-1}(\bm{x})u(r,\bm{x})\,.
\end{equation*}
This function satisfies
\begin{equation*}
	\begin{cases}
		\partial_{r}v(r,\bm{x})= (\mathcal{A}_{N}v(r,\cdot))(\bm{x})\\
		v(0,\bm{x})=1
	\end{cases}
\end{equation*}
where
\begin{equation*}
	(\mathcal{A}_{N} f)(\bm{x})
	:=
	h_{N}^{-1}(\bm{x})
	\left((L^{(2)}-c_{N})(h_{N}f)\right)(\bm{x}).
\end{equation*}
Let
\begin{equation*}
	V_{N}(\bm{x})=h_{N}^{-1}(\bm{x}) \left(\left(L^{(2)}-L^{\text{ind}}\right)h_{N}\right)(\bm{x})\,.
\end{equation*}
This function is actually a function of the relative distance only, namely $V_{N}=V_{N}(d)$ with $d=x_{1}-x_{2}$ and it reads
\begin{equation*}
	V_{N}(d)=4b \cosh(N^{-1/4})(\cosh(N^{-1/4})-1)\mathbbm{1}_{\{d=0\}}	-2a (\cosh(N^{-1/4})-1)\mathbbm{1}_{\{|d|=1\}}\,,
\end{equation*}
where 
\begin{equation*}
	a= \mathbb{E}[B(1-B)]=\frac{\alpha}{2(2\alpha+1)},\qquad b= \mathbb{E}[B^{2}]=\frac{\alpha+1}{2(2\alpha+1)}\,.
\end{equation*}
Therefore by using $\cosh(N^{-1/4})-1\leq \frac{1}{2}\cosh(1)N^{-1/2}$, one has that
\begin{equation}\label{estimate-FK}
	V_{N}(d)\leq N^{-1/2} C_{\alpha}\mathbbm{1}_{\{d=0\}}\,.
\end{equation}
We now decompose the operator  $\mathcal{A}_{N}$ as 
\begin{equation*}
	\mathcal{A}_{N}=	\widehat{L}_{N}+V_{N}(\bm{x})\,.
\end{equation*}
Here $\widehat{L}_{N}$ is Markov generator with transition rates 
\begin{equation*}
	\widehat{q}^{(2)}(\bm{x},\bm{y})= q^{(2)}(\bm{x},\bm{y})\frac{h_{N}(\bm{y})}{h_{N}(\bm{x})}\,.
\end{equation*}
This is a well defined generator and it defines a Markov chain  $(\widehat{Y}(r))_{r\geq 0}$, with path space measure $\widehat{\mathbf{P}}$ and expectation $\widehat{\mathbf{E}}$. Furthermore, we define $(D_{s})_{s\geq 0}$ is the difference Markov chain, namely $D_{r}=\widehat{Y}_{1}(r)-\widehat{Y}_{2}(r)$ for all $r\geq 0$. 
The Kolmogorov equation for $v(r,\bm{x})$ reads 
\begin{equation*}
	\begin{cases}
		\partial_{r}v(r,\bm{x})= ((\widehat{L}_{N}v(r,\cdot))(\bm{x})+V_{N}(\bm{x})v(r,\bm{x})\\
		v(0,\bm{x})=1
	\end{cases}\,.
\end{equation*}
By applying Feynman-Kac formula, we have that 
\begin{equation}
	v(r,\bm{x})= \widehat{\mathbf{E}}_{\bm{x}}\left[\exp{\left(\int_{0}^{r}V_{N}(\widehat{\bm{Y}}(s))ds\right)}\right]\leq \widehat{\mathbf{E}}_{\bm{x}}\left[\exp{\left(C_{\alpha}N^{-1/2}\int_{0}^{r}\mathbbm{1}_{\{D_{s}=0\}}ds\right)}\right]\,.
	\label{eq:ulessthanLaplacelocaltime}
\end{equation}
where in the inequality we used \eqref{estimate-FK}. 
Finally, setting $r=Nt$ and using $h_N(0,0)=1$, we obtain
\begin{equation*}
	\begin{split}
		\mathbb E_\omega[(Z_t^{(N)})^2]
		\leq
		\widehat{\mathbf E}_{(0,0)}
		\left[
		e^{C_\alpha N^{-1/2}\mathbb{L}_{0}^{Nt}(D)}
		\right]\,,
	\end{split}
\end{equation*}
where we have introduced the local time at $0$ of the difference process $D_{s}$ under the path-space measure $\widehat{\mathbf{P}}$
\begin{equation*}
	\mathbb{L}_{0}^{Nt}(D):=\int_{0}^{tN}\mathbbm{1}_{\{D_{s}=0\}}ds\,.
\end{equation*}

\subsection{Estimation of the local time of $D_t$}
To conclude the proof of Lemma \ref{Lemma5-2}, we need to estimate the local time of the process $d_t$. First we reduce the problem to estimating the local time of a continuous time symmetric random walk. 

Let consider the process $(\mathsf{Y}_{t})_{t\geq 0}$ consisting of a continuous time random walk with transition rates
\begin{itemize}
	\item $0 \to \pm 1$ at rate $2 a \cosh(\lambda)$, 
	\item $\pm 1 \to 0$ at rate $\cosh(\lambda)$, 
	\item $d \to d\pm 1$ at rate $\cosh(\lambda)$ for all other values (including $1\to 2$ and $2\to 1$).
\end{itemize}
When changing $D_t$ to $Y_t$, we increase the rate of transitions $\pm 1\to 0$, so this can only increase the local time, $\mathbb{L}_{t}^{0}(D) \leq \mathbb{L}_{t}^{0}(\mathsf{Y})$, so that 
\begin{equation*}
	\widehat{\mathbf{E}}_{0}\left[e^{C_{\alpha}N^{-1/2}\mathbb{L}_{t}^{0}(D)}\right]
	\leq \mathbf{E}^{\mathsf{Y}}\left[e^{C_{\alpha}N^{-1/2}\mathbb{L}_{t}^{0}(\mathsf{Y})}\right]\,.
\end{equation*}

We further introduce a process $(\mathsf{X}_{t})_{t\geq 0}$ that has all rates equal to $\cosh(\lambda)$. We then construct a coupling, such that the two processes has the same skeleton chain, denoted by $(S_{n})_{n\in \mathbb{Z}_{\geq 0}}$ (and consisting of a simple symmetric random walk), the same waiting times out of $0$, 
but the process $\mathsf{Y}$ leaves $0$ according to a waiting time $H_{n}^{\mathsf{Y}}\sim\exp(4a \cosh(\lambda))$ while the process $\mathsf{X}$ leaves $0$ with a waiting time $H_{n}^{\mathsf{X}}\sim \exp(2\cosh(\lambda))$. This means that $H_{n}^{\mathsf{Y}}= \frac{1}{2a}H_{n}^{\mathsf{X}}$. 
Assume that both processes start from $0$ and introduce $Q_{\mathsf{X},\mathsf{Y}}$ the joint probability measure, with expectation denoted by $E^{Q_{X,Y}}$ such that the marginals correspond to the laws of $\mathsf{X}$ and $\mathsf{Y}$ respectively. 

Denote by $(T_{n}^{\mathsf{Y}})_{n\in\mathbb{Z}_{\geq 0}}$ and $(T_{n}^{\mathsf{X}})_{n\in\mathbb{Z}_{\geq 0}}$ the jump times of the $\mathsf{Y}$ and $\mathsf{X}$ respectively, then we have that 
\begin{equation*}
	\begin{split}
		\mathsf{X}_{t}= S_{n}\qquad \text{if}\quad T_{n}^{\mathsf{X}}\leq t< T_{n+1}^{\mathsf{X}}\\
		\mathsf{Y}_{\tau}= S_{n}\qquad \text{if}\quad T_{n}^{\mathsf{Y}}\leq \tau< T_{n+1}^{\mathsf{Y}}
	\end{split}
\end{equation*}
Therefore, fix $n\in \mathbb{Z}_{\geq 0}$ consider the times when both processes are in $S_{n}$. We have that, for $0\leq v< H_{n}$,
\begin{equation*}
	t= T_{n}^{\mathsf{X}}+v
\end{equation*}
and 
\begin{equation*}
	\tau= T_{n}^{\mathsf{Y}}+
	\begin{cases}
		v\quad &S_{n}\neq 0\\
		\frac{v}{2a}\quad& S_{n}=0
	\end{cases}\,.
\end{equation*}
Then, we have that 
\begin{equation*}
	\tau(t)=t+\left(\frac{1}{2a}-1\right)\mathbb{L}_{0}^{t}(\mathsf{X})\,.
\end{equation*}
This also implies that 
\begin{equation}\label{connection-local}
	\mathbb{L}_{0}^{\tau(t)}(\mathsf{Y})=\frac{1}{2a}\mathbb{L}_{0}^{t}(\mathsf{X})
	\qquad Q_{\mathsf{X},\mathsf{Y}}-\text{a.s.}
\end{equation}
Finally, since $\tau(t)\geq t$ we have that 
\begin{equation}\label{upper-local}
	\mathbb{L}_{0}^{t}(\mathsf{Y})\leq \mathbb{L}_{0}^{\tau(t)}(\mathsf{Y})
	= \frac{1}{2a}\mathbb{L}_{0}^{t}(\mathsf{X})
	\qquad Q_{\mathsf{X},\mathsf{Y}}-\text{a.s.}
\end{equation}
Let $E^{\mathsf{X}}[\cdot]$ and $E^{\mathsf{Y}}[\cdot]$ the expectations with respect to the laws of $\mathsf{X}$ and $\mathsf{Y}$ respectively. From the definition of coupling and from equations \eqref{connection-local} and \eqref{upper-local}, it follows that, for any $\beta >0$, 
\begin{equation*}
	\begin{aligned}
		E^{\mathsf{Y}}[e^{\beta \mathbb{L}_{0}^{t}(\mathsf{Y})}]
		&= E^{Q_{\mathsf{X},\mathsf{Y}}}[e^{\beta \mathbb{L}_{0}^{t}(\mathsf{Y})}]\\
		&\leq E^{Q_{\mathsf{X},\mathsf{Y}}}[e^{\beta \mathbb{L}_{0}^{\tau(t)}(\mathsf{Y})}]\\
		&= E^{Q_{\mathsf{X},\mathsf{Y}}}[e^{\frac{\beta}{2a}\mathbb{L}_{0}^{t}(\mathsf{X})}]
		= E^{\mathsf{X}}[e^{\frac{\beta}{2a}\mathbb{L}_{0}^{t}(\mathsf{X})}]\,.
	\end{aligned}
\end{equation*}
Therefore, we obtain the estimate \eqref{local-time_D-X}. 

\subsection{Details to bound the local time of $\mathsf{X}$}
Let $t^{*}:= \frac{2a^{2}}{\pi C_{\alpha}^{2}}$, implying that for all $t\leq t^{*}$ we have that $E_{0}^{\mathsf{X}}[\frac{C_{\alpha}N^{-1/2}}{2a}\mathbb{L}_{0}^{Nt}]\leq \frac{1}{2}<1$. Then, for all $t<t^{*}$ by  Khas’minskii Lemma \cite{Khasminskii1959} and by the inequality $(1-u)^{-1}\leq 1+2u$ for $u\in[0,1/2]$, we have that  
\begin{equation*}
	E_{0}^{\mathsf{X}}\left[e^{\frac{C_{\alpha}N^{-1/2}}{2a}\mathbb{L}_{0}^{Nt}}\right]\leq \frac{1}{1-\frac{C_{\alpha}N^{-1/2}}{2a}E_{0}^{\mathsf{X}}[\mathbb{L}_{0}^{Nt}]}\leq  1+C\sqrt{t}\,.
\end{equation*}
We now extend to $t^{*}\leq t\leq T$. First, we show that for all $x\in \mathbb{Z}$, $\frac{C_{\alpha}N^{-1/2}}{2a}\geq 0$ and $t>0$ we have that 
\begin{equation}\label{worst-0}
	E_{x}^{\mathsf{X}}\left[e^{\frac{C_{\alpha}N^{-1/2}}{2a} \mathbb{L}^{t}_{0}(\mathsf{X})}\right]\leq E_{0}^{\mathsf{X}}\left[e^{\frac{C_{\alpha}N^{-1/2}}{2a} \mathbb{L}^{t}_{0}(\mathsf{X})}\right] \,.
\end{equation}
To this aim, let $\sigma_{0}:=\inf\left\{s\geq 0\;:\; \mathsf{X}_{s}=0\right\}$. Then, we by the law of total expectation we have that 
\begin{equation*}
	\begin{split}
		E_{x}^{\mathsf{X}}\left[e^{\frac{C_{\alpha}N^{-1/2}}{2a} \mathbb{L}_{0}^{t}}(\mathsf{X})\right]&= E_{x}^{\mathsf{X}}\left[e^{\frac{C_{\alpha}N^{-1/2}}{2a} \mathbb{L}_{0}^{t}}(\mathsf{X})\bigg|\{\sigma_{0}>t\}\right]P_{x}^{\mathsf{X}}(\{\sigma_{0}>t\})
		\\&\quad+E_{x}^{\mathsf{X}}\left[e^{\frac{C_{\alpha}N^{-1/2}}{2a} \mathbb{L}_{0}^{t}}(\mathsf{X})\bigg|\{\sigma_{0}\leq t\}\right]P_{x}^{\mathsf{X}}(\{\sigma_{0}\leq t\})
		\\&=
		P_{x}^{\mathsf{X}}(\{\sigma_{0}>t\})+E_{x}^{\mathsf{X}}\left[\mathbbm{1}_{\{\sigma_{0}\leq t\}}e^{\frac{C_{\alpha}N^{-1/2}}{2a} \mathbb{L}_{0}^{t-\sigma_{0}}}(\mathsf{X})\right]
		\\&\leq E_{0}^{\mathsf{X}}\left[e^{\frac{C_{\alpha}N^{-1/2}}{2a} \mathbb{L}_{0}^{t}}(\mathsf{X})\right]\,.
	\end{split}
\end{equation*}
Let us denote $h:= N t^{*}$, then as a consequence of \eqref{worst-0}, we have that
\begin{equation*}
	E_{x}^{\mathsf{X}}\left[e^{\frac{C_{\alpha}N^{-1/2}}{2a} \mathbb{L}^{h}_{0}(\mathsf{X})}\right]\leq E_{0}^{\mathsf{X}}\left[e^{\frac{C_{\alpha}N^{-1/2}}{2a} \mathbb{L}^{h}_{0}(\mathsf{X})}\right]\leq 2\,.
\end{equation*}
For all $j\in \mathbb{Z}_{\geq 1}$ we write
\begin{equation*}
	\mathbb{L}_{0}^{hj}(\mathsf{X})=\mathbb{L}_{0}^{(j-1)h}(\mathsf{X})+\int_{(j-1)h}^{jh}\mathbbm{1}_{\{\mathsf{X}_{s}=0\}}ds\,.
\end{equation*}
Furthermore, we define $K:=\lceil T/t^{*} \rceil$ and we introduce the natural filtration of the process $\mathsf{X}_{s}$ denoted by $\left(\mathcal{F}_{s}\right)_{s\geq 0}$. Therefore, we have that 
\begin{equation*}
	\begin{split}
		E_{0}^{\mathsf{X}}\left[e^{\frac{C_{\alpha}N^{-1/2}}{2a}\mathbb{L}_{0}^{Nt}}\right]\leq &	E_{0}^{\mathsf{X}}\left[e^{\frac{C_{\alpha}N^{-1/2}}{2a}\mathbb{L}_{0}^{Kh}}\right]
		\\\leq& 	E_{0}^{\mathsf{X}}\left[E_{0}^{\mathsf{X}}\left[e^{\frac{C_{\alpha}N^{-1/2}}{2a}\mathbb{L}_{0}^{Kh}}\bigg|\mathcal{F}_{(K-1)h}\right]\right]
		\\\leq&
		2E_{0}^{\mathsf{X}}\left[e^{\frac{C_{\alpha}N^{-1/2}}{2a}\mathbb{L}_{0}^{(K-1)h}}\right]\leq \ldots \leq 2^{K}
		\\\leq & 
		1+2^{K}\sqrt{t/t^{*}}
		\\\leq &
		1+ C\sqrt{t}\,.
	\end{split}
\end{equation*}
}

\subsection*{Data availability}
No new data were created or analyzed in this study.
\subsection*{Conflict of interest}
The authors have no relevant financial or non-financial interests to disclose. 

\printbibliography

@article{kardar1992,
  title={Directed waves in random media},
  author={Saul, Lawrence and Kardar, Mehran and Read, N},
  journal={Physical Review A},
  volume={45},
  number={12},
  pages={8859},
  year={1992},
  publisher={APS}
}

@article{barraquand2015,
  title={Random-walk in beta-distributed random environment},
  author={Barraquand, Guillaume and Corwin, Ivan},
  journal={Probability Theory and Related Fields},
  volume={167},
  number={3},
  pages={1057--1116},
  year={2017},
  publisher={Springer}
}

@article{barraquand2022,
  title={Random walk on nonnegative integers in beta distributed random environment},
  author={Barraquand, Guillaume and Rychnovsky, Mark},
  journal={Communications in Mathematical Physics},
  volume={398},
  number={2},
  pages={823--875},
  year={2023},
  publisher={Springer}
}

@article{kipnis1982heat,
  title={Heat flow in an exactly solvable model},
  author={Kipnis, Claude and Marchioro, Carlo and Presutti, Errico},
  journal={Journal of Statistical Physics},
  volume={27},
  pages={65--74},
  year={1982},
  publisher={Springer}
}

@article{modelsOfTransport,
  title={Duality for stochastic models of transport},
  author={Carinci, Gioia and Giardin{\`a}, Cristian and Giberti, Claudio and Redig, Frank},
  journal={Journal of Statistical Physics},
  volume={152},
  pages={657--697},
  year={2013},
  publisher={Springer}
}

@article{hiddenSymetries,
  title={Duality and hidden symmetries in interacting particle systems},
  author={Giardina, Cristian and Kurchan, Jorge and Redig, Frank and Vafayi, Kiamars},
  journal={Journal of Statistical Physics},
  volume={135},
  pages={25--55},
  year={2009},
  publisher={Springer}
}

@Article{das2024multiplicative,
  author    = {Das, Sayan and Drillick, Hindy and Parekh, Shalin},
  journal   = {Probability Theory and Related Fields},
  title     = {Multiplicative {SHE} limit of random walks in space--time random environments},
  year      = {2024},
  pages     = {1--83},
  publisher = {Springer},
}

@article{parekh2,
  title={Hierarchy of {KPZ} limits arising from directed random walk models in random media},
  author={Parekh, Shalin},
  journal={arXiv:2401.06073},
  year={2024}
}

@book{landim-book,
  title={Scaling limits of interacting particle systems},
  author={Kipnis, Claude and Landim, Claudio},
  volume={320},
  year={2013},
  publisher={Springer Science \& Business Media}
}

@book{brownianWeb,
  title={Stochastic flows in the {Brownian} web and net},
  author={Schertzer, Emmanuel and Sun, Rongfeng and Swart, Jan},
  volume={227},
  number={1065},
  year={2014},
  publisher={American mathematical society}
}

@book{billingsley,
  title={Convergence of probability measures},
  author={Billingsley, Patrick},
  year={1999},
  publisher={John Wiley \& Sons}
}

@article{aldous,
  title={Stopping times and tightness},
  author={Aldous, David},
  journal={The Annals of Probability},
  pages={335--340},
  year={1978},
  publisher={JSTOR}
}

@article{KPZ-original,
  title={Dynamic scaling of growing interfaces},
  author={Kardar, Mehran and Parisi, Giorgio and Zhang, Yi-Cheng},
  journal={Physical Review Letters},
  volume={56},
  number={9},
  pages={889},
  year={1986},
  publisher={APS}
}

@article{parekh2019-BD,
  title={The {KPZ limit of ASEP} with boundary},
  author={Parekh, Shalin},
  journal={Communications in Mathematical Physics},
  volume={365},
  pages={569--649},
  year={2019},
  publisher={Springer}
}

@article{B-LD-moderate,
  title={Moderate deviations for diffusion in time dependent random media},
  author={Barraquand, Guillaume and Le Doussal, Pierre},
  journal={Journal of Physics A: Mathematical and Theoretical},
  volume={53},
  number={21},
  pages={215002},
  year={2020},
  publisher={IOP Publishing}
}

@article{leDoussal-17,
  title={Diffusion in time-dependent random media and the {Kardar-Parisi-Zhang} equation},
  author={Le Doussal, Pierre and Thiery, Thimoth{\'e}e},
  journal={Physical Review E},
  volume={96},
  number={1},
  pages={010102},
  year={2017},
  publisher={APS}
}

@article{hass2025universal,
  title={Universal {KPZ} Fluctuations for Moderate Deviations of Random Walks in Random Environments},
  author={Hass, Jacob and Drillick, Hindy and Corwin, Ivan and Corwin, Eric},
  journal={arXiv preprint arXiv:2504.00266},
  year={2025}
}

@article{hass2024extreme,
  title = {Extreme Diffusion Measures Statistical Fluctuations of the Environment},
  author = {Hass, Jacob B. and Drillick, Hindy and Corwin, Ivan and Corwin, Eric I.},
  journal = {Phys. Rev. Lett.},
  volume = {133},
  issue = {26},
  pages = {267102},
  numpages = {8},
  year = {2024},
  month = {Dec},
  publisher = {American Physical Society},
  doi = {10.1103/PhysRevLett.133.267102},
  url = {https://link.aps.org/doi/10.1103/PhysRevLett.133.267102}
}

@article{hass2024first,
  title = {First-passage time for many-particle diffusion in space-time random environments},
  author = {Hass, Jacob B. and Corwin, Ivan and Corwin, Eric I.},
  journal = {Phys. Rev. E},
  volume = {109},
  issue = {5},
  pages = {054101},
  numpages = {15},
  year = {2024},
  month = {May},
  publisher = {American Physical Society},
  doi = {10.1103/PhysRevE.109.054101},
  url = {https://link.aps.org/doi/10.1103/PhysRevE.109.054101}
}

@article{mitoma,
  title={Tightness of probabilities on {C ([0, 1]; Y') and D ([0, 1]; Y')}},
  author={Mitoma, Itaru},
  journal={The Annals of Probability},
  pages={989--999},
  year={1983},
  publisher={JSTOR}
}

@article{de2024hidden,
  title={Hidden temperature in the {KMP} model},
  author={De Masi, Anna and Ferrari, Pablo A and Gabrielli, Davide},
  journal={Journal of Statistical Physics},
  volume={191},
  number={11},
  pages={150},
  year={2024},
  publisher={Springer}
}

@article{giardina2025intertwining,
  title={Intertwining and propagation of mixtures for generalized {KMP} models and harmonic models},
  author={Giardin{\`a}, Cristian and Redig, Frank and van Tol, Berend},
  journal={Journal of Statistical Physics},
  volume={192},
  number={2},
  pages={21},
  year={2025},
  publisher={Springer}
}

@article{franceschini2025hydrodynamic,
  title={Hydrodynamic limit for some gradient and attractive spin models},
  author={Franceschini, Chiara and Gon{\c{c}}alves, Patr{\'\i}cia and Hayashi, Kohei and Sasada, Makiko},
  journal={arXiv preprint arXiv:2505.11092},
  year={2025}
}

@article{bettelheim2022full,
  title={Full statistics of nonstationary heat transfer in the {Kipnis--Marchioro--Presutti} model},
  author={Bettelheim, Eldad and Smith, Naftali R and Meerson, Baruch},
  journal={Journal of Statistical Mechanics: Theory and Experiment},
  volume={2022},
  number={9},
  pages={093103},
  year={2022},
  publisher={IOP Publishing}
}

@article{das2024kpz,
  title={{KPZ equation limit of sticky Brownian motion}},
  author={Das, Sayan and Drillick, Hindy and Parekh, Shalin},
  journal={Journal of Functional Analysis},
  volume={287},
  number={10},
  pages={110609},
  year={2024},
  publisher={Elsevier}
}

@article{10.1214/EJP.v3-28,
author = {P.A. Ferrari and L. Fontes},
title = {Fluctuations of a Surface Submitted to a Random Average Process},
volume = {3},
journal = {Electronic Journal of Probability},
number = {none},
publisher = {Institute of Mathematical Statistics and Bernoulli Society},
pages = {1 -- 34},
year = {1998},
doi = {10.1214/EJP.v3-28},
URL = {https://doi.org/10.1214/EJP.v3-28}
}

@article{grabsch2022exact,
  title={Exact closure and solution for spatial correlations in single-file diffusion},
  author={Grabsch, Aur{\'e}lien and Poncet, Alexis and Rizkallah, Pierre and Illien, Pierre and B{\'e}nichou, Olivier},
  journal={Science Advances},
  volume={8},
  number={12},
  pages={eabm5043},
  year={2022},
  publisher={American Association for the Advancement of Science}
}

@article{rizkallah2023duality,
  title={Duality relations in single-file diffusion},
  author={Rizkallah, Pierre and Grabsch, Aur{\'e}lien and Illien, Pierre and B{\'e}nichou, Olivier},
  journal={Journal of Statistical Mechanics: Theory and Experiment},
  volume={2023},
  number={1},
  pages={013202},
  year={2023},
  publisher={IOP Publishing}
}

@article{balazs2006random,
  title={The random average process and random walk in a space-time random environment in one dimension},
  author={Bal{\'a}zs, M{\'a}rton and Rassoul-Agha, Firas and Sepp{\"a}l{\"a}inen, Timo},
  journal={Communications in mathematical physics},
  volume={266},
  number={2},
  pages={499--545},
  year={2006},
  publisher={Springer}
}

@article{le2004products,
  title={Products of {Beta} matrices and sticky flows},
  author={Le Jan, Yves and Lemaire, Sophie},
  journal={Probability Theory and Related Fields},
  volume={130},
  pages={109--134},
  year={2004},
  publisher={Springer}
}

@article{le2013markovian,
  title={Markovian loop clusters on graphs},
  author={Le Jan, Yves and Lemaire, Sophie},
  journal={Illinois Journal of Mathematics},
  volume={57},
  number={2},
  pages={525--558},
  year={2013},
  publisher={Duke University Press}
}

@article{jan2004sticky,
  title={Sticky flows on the circle and their noises},
  author={Jan, Yves Le and Raimond, Olivier},
  journal={Probability Theory and Related Fields},
  volume={129},
  number={1},
  pages={63--82},
  year={2004},
  publisher={Springer}
}

@article{howitt2009consistent,
author = {Howitt, Chris and Warren, Jon},
year = {2006},
month = {11},
pages = {},
title = {Consistent families of {Brownian} motions and stochastic flows of kernels.},
volume = {37},
journal = {The Annals of Probability},
doi = {10.1214/08-AOP431}
}

@article{brockington2022edge,
  title={At the edge of a cloud of {Brownian} particles},
  author={Brockington, Dom and Warren, Jon},
  journal={arXiv preprint arXiv:2208.11952},
  year={2022}
}

@article{yu2016edwards,
  title={{Edwards--Wilkinson fluctuations in the Howitt--Warren flows}},
  author={Yu, Jinjiong},
  journal={Stochastic Processes and their Applications},
  volume={126},
  number={3},
  pages={948--982},
  year={2016},
  publisher={Elsevier}
}

@article{barraquand2020large,
author = {Guillaume Barraquand and Mark Rychnovsky},
title = {{Large deviations for sticky Brownian motions}},
volume = {25},
journal = {Electronic Journal of Probability},
publisher = {Institute of Mathematical Statistics and Bernoulli Society},
pages = {1 -- 52},
year = {2020},
doi = {10.1214/20-EJP515},
URL = {https://doi.org/10.1214/20-EJP515}
}

@article{fisher2023random,
  title={Random quantum circuits},
  author={Fisher, Matthew PA and Khemani, Vedika and Nahum, Adam and Vijay, Sagar},
  journal={Annual Review of Condensed Matter Physics},
  volume={14},
  number={1},
  pages={335--379},
  year={2023},
  publisher={Annual Reviews}
}

@article{nahum2018operator,
  title={Operator spreading in random unitary circuits},
  author={Nahum, Adam and Vijay, Sagar and Haah, Jeongwan},
  journal={Physical Review X},
  volume={8},
  number={2},
  pages={021014},
  year={2018},
  publisher={APS}
}

@article{bertini1997stochastic,
  title={{Stochastic Burgers and KPZ} equations from particle systems},
  author={Bertini, Lorenzo and Giacomin, Giambattista},
  journal={Communications in mathematical physics},
  volume={183},
  pages={571--607},
  year={1997},
  publisher={Springer}
}

@article{jansen2014notion,
  title={On the notion (s) of duality for {Markov} processes},
  author={Jansen, Sabine and Kurt, Noemi},
  journal={arXiv:1210.7193},
  year={2012}
}

@article{bettelheim2024full,
  title={Full statistics of regularized local energy density in a freely expanding {Kipnis--Marchioro--Presutti} gas},
  author={Bettelheim, Eldad and Meerson, Baruch},
  journal={Journal of Statistical Mechanics: Theory and Experiment},
  volume={2024},
  number={11},
  pages={113204},
  year={2024},
  publisher={IOP Publishing}
}

@article{bettelheim2022inverse,
  title={Inverse scattering method solves the problem of full statistics of nonstationary heat transfer in the {Kipnis-Marchioro-Presutti} model},
  author={Bettelheim, Eldad and Smith, Naftali R and Meerson, Baruch},
  journal={Physical Review Letters},
  volume={128},
  number={13},
  pages={130602},
  year={2022},
  publisher={APS}
}

@article{krajenbrink2023crossover,
  title={Crossover from the macroscopic fluctuation theory to the {Kardar-Parisi-Zhang} equation controls the large deviations beyond Einstein's diffusion},
  author={Krajenbrink, Alexandre and Le Doussal, Pierre},
  journal={Physical Review E},
  volume={107},
  number={1},
  pages={014137},
  year={2023},
  publisher={APS}
}

@article{bertini2015macroscopic,
  title={Macroscopic fluctuation theory},
  author={Bertini, Lorenzo and De Sole, Alberto and Gabrielli, Davide and Jona-Lasinio, Giovanni and Landim, Claudio},
  journal={Reviews of Modern Physics},
  volume={87},
  number={2},
  pages={593--636},
  year={2015},
  publisher={APS}
}

@article{bertini2002macroscopic,
  title={Macroscopic fluctuation theory for stationary non-equilibrium states},
  author={Bertini, Lorenzo and De Sole, Alberto and Gabrielli, Davide and Jona-Lasinio, Giovanni and Landim, Claudio},
  journal={Journal of Statistical Physics},
  volume={107},
  number={3},
  pages={635--675},
  year={2002},
  publisher={Springer}
}

@article{thiery2015integrable,
  title={On integrable directed polymer models on the square lattice},
  author={Thiery, Thimoth{\'e}e and Le Doussal, Pierre},
  journal={Journal of Physics A: Mathematical and Theoretical},
  volume={48},
  number={46},
  pages={465001},
  year={2015},
  publisher={IOP Publishing}
}

@article{thiery2016exact,
  title={Exact solution for a random walk in a time-dependent {1D} random environment: the point-to-point {Beta polymer}},
  author={Thiery, Thimoth{\'e}e and Le Doussal, Pierre},
  journal={Journal of Physics A: Mathematical and Theoretical},
  volume={50},
  number={4},
  pages={045001},
  year={2016},
  publisher={IOP Publishing}
}

@article{hass2023anomalous,
  title={Anomalous fluctuations of extremes in many-particle diffusion},
  author={Hass, Jacob B and Carroll-Godfrey, Aileen N and Corwin, Ivan and Corwin, Eric I},
  journal={Physical Review E},
  volume={107},
  number={2},
  pages={L022101},
  year={2023},
  publisher={APS}
}

@article{corwin2017kardar,
  title={{Kardar--Parisi--Zhang} equation and large deviations for random walks in weak random environments},
  author={Corwin, Ivan and Gu, Yu},
  journal={Journal of Statistical Physics},
  volume={166},
  number={1},
  pages={150--168},
  year={2017},
  publisher={Springer}
}

@article{hass2025super,
  title={Super-Universal Behavior of Outliers Diffusing in a Space-Time Random Environment},
  author={Hass, Jacob},
  journal={arXiv preprint arXiv:2505.01533},
  year={2025}
}

@article{matetski2021kpz,
  title={The{ KPZ} fixed point},
  author={Matetski, Konstantin and Quastel, Jeremy and Remenik, Daniel},
  journal={Acta Mathematica},
  volume={227},
  number={1},
  pages={115--203},
  year={2021},
  publisher={Lehigh University Bethlehem, Penn., USA}
}

@article{zarfaty2016statistics,
  title={Statistics of large currents in the {Kipnis--Marchioro--Presutti} model in a ring geometry},
  author={Zarfaty, Lior and Meerson, Baruch},
  journal={Journal of Statistical Mechanics: Theory and Experiment},
  volume={2016},
  number={3},
  pages={033304},
  year={2016},
  publisher={IOP Publishing}
}

@article{bertini1995stochastic,
  title={The stochastic heat equation: {Feynman-Kac} formula and intermittence},
  author={Bertini, Lorenzo and Cancrini, Nicoletta},
  journal={Journal of statistical Physics},
  volume={78},
  number={5},
  pages={1377--1401},
  year={1995},
  publisher={Springer}
}

@article{brockington2023bethe,
  title={The {Bethe} ansatz for sticky {Brownian} motions},
  author={Brockington, Dom and Warren, Jon},
  journal={Stochastic Processes and their Applications},
  volume={162},
  pages={1--48},
  year={2023},
  publisher={Elsevier}
}

@article{carinci2016asymmetric,
  title={Asymmetric stochastic transport models with ${U}_{q} (su (1, 1))$ symmetry},
  author={Carinci, Gioia and Giardina, Cristian and Redig, Frank and Sasamoto, Tomohiro},
  journal={Journal of Statistical Physics},
  volume={163},
  number={2},
  pages={239--279},
  year={2016},
  publisher={Springer}
}

@article{kardar1994directed,
  title={Directed paths in random media},
  author={Kardar, Mehran},
  journal={arXiv preprint cond-mat/9411022},
  year={1994}
}

@article{PhysRevB.50.5119,
  title = {Scaling phenomena in a unitary model of directed propagating waves with applications to one-dimensional electrons in a time-varying potential},
  author = {Cule, Dinko and Shapir, Yonathan},
  journal = {Phys. Rev. B},
  volume = {50},
  issue = {8},
  pages = {5119--5130},
  numpages = {0},
  year = {1994},
  month = {Aug},
  publisher = {American Physical Society},
  doi = {10.1103/PhysRevB.50.5119},
  url = {https://link.aps.org/doi/10.1103/PhysRevB.50.5119}
}

@Misc{ledoussal2023private,
  author = {Le Doussal, P.},
  title  = {private communication},
  year   = {2023},
}

@Article{balazs2018large,
  author        = {Balázs, Márton and Rassoul-Agha, Firas and Seppäläinen, Timo},
  journal       = {Ann. Probab. 47(4): 2186-2229 (July 2019)},
  title         = {Large deviations and wandering exponent for random walk in a dynamic beta environment},
  year          = {2018},
  issn          = {0091-1798},
  month         = jul,
  number        = {4},
  volume        = {47},
  archiveprefix = {arXiv},
  copyright     = {arXiv.org perpetual, non-exclusive license},
  date          = {2018-01-24},
  doi           = {10.1214/18-aop1306},
  eprint        = {1801.08070},
  primaryclass  = {math.PR},
  publisher     = {Institute of Mathematical Statistics},
}

@book{dualityBook,
	title={Duality for {M}arkov processes: a Lie algebraic approach},
	author={Giardinà, Cristian and Redig, Frank},
	year={2025},
	publisher={To appear on  Springer Cham, Grundlehren der mathematischen Wissenschaften (GL, volume 365), Series E-ISSN: 978-3-032-04098-5 }
}

@Article{tracy1992level,
  author        = {Tracy, Craig A. and Widom, Harold},
  journal       = {Commun.Math.Phys. 159 (1994) 151-174},
  title         = {Level-Spacing Distributions and the Airy Kernel},
  year          = {1992},
  issn          = {1432-0916},
  month         = jan,
  number        = {1},
  pages         = {151--174},
  volume        = {159},
  archiveprefix = {arXiv},
  copyright     = {Assumed arXiv.org perpetual, non-exclusive license to distribute this article for submissions made before January 2004},
  date          = {1992-12-01},
  doi           = {10.1007/bf02100489},
  eprint        = {hep-th/9211141},
  primaryclass  = {hep-th},
  publisher     = {Springer Science and Business Media LLC},
}

@Article{drillick2025random,
  author        = {Drillick, Hindy and Parekh, Shalin},
  title         = {Random walks in space-time random media in all spatial dimensions: the full subcritical fluctuation regime},
  year          = {2025},
  month         = oct,
  archiveprefix = {arXiv},
  copyright     = {Creative Commons Attribution 4.0 International},
  doi           = {10.48550/ARXIV.2510.22155},
  eprint        = {2510.22155},
  primaryclass  = {math.PR},
  publisher     = {arXiv},
}

@Article{alberts2014intermediate,
  author        = {Tom Alberts and Konstantin Khanin and Jeremy Quastel},
  journal       = {The Annals of Probability},
  title         = {The intermediate disorder regime for directed polymers in dimension $1+ 1$},
  year          = {2014},
  issn          = {00911798},
  number        = {3},
  pages         = {1212--1256},
  volume        = {42},
  archiveprefix = {arXiv},
  eprint        = {1202.4398},
  publisher     = {Institute of Mathematical Statistics},
  url           = {http://www.jstor.org/stable/42919861},
  urldate       = {2025-07-12},
}

@Article{BoldrighiniIgnatyukMalyshevPellegrinotti1992,
  author  = {Boldrighini, C. and Ignatyuk, I. A. and Malyshev, V. A. and Pellegrinotti, A.},
  journal = {Stochastic Processes and their Applications},
  title   = {Random walk in dynamic environment with mutual influence},
  year    = {1992},
  number  = {1},
  pages   = {157--177},
  volume  = {41},
  doi     = {10.1016/0304-4149(92)90151-F},
}

@Article{BoldrighiniMinlosPellegrinotti1997,
  author  = {Boldrighini, C. and Minlos, R. A. and Pellegrinotti, A.},
  journal = {Probability Theory and Related Fields},
  title   = {Almost-sure central limit theorem for a {M}arkov model of random walk in dynamical random environment},
  year    = {1997},
  number  = {2},
  pages   = {245--273},
  volume  = {109},
  doi     = {10.1007/s004400050132},
}

@Article{BernabeiBoldrighiniMinlosPellegrinotti1998,
  author  = {Bernabei, M. S. and Boldrighini, C. and Minlos, R. A. and Pellegrinotti, A.},
  journal = {Markov Processes and Related Fields},
  title   = {Almost-sure central limit theorem for a model of random walk in fluctuating random environment},
  year    = {1998},
  number  = {3},
  pages   = {381--393},
  volume  = {4},
  url     = {https://math-mprf.org/journal/articles/id818/},
}

@Article{boldrighini1999central,
  author  = {Boldrighini, C. and Minlos, R. A. and Pellegrinotti, A.},
  journal = {Theory of Stochastic Processes},
  title   = {Central limit theorem for a random walk in dynamical environment: integral and local},
  year    = {1999},
  number  = {3--4},
  pages   = {16--28},
  volume  = {5(21)},
}

@Article{boldrighini2001t,
  author  = {Boldrighini, C. and Pellegrinotti, A.},
  journal = {Moscow Mathematical Journal},
  title   = {$T^{-1/4}$-noise for random walks in dynamic environment on $\mathbb{Z}$},
  year    = {2001},
  number  = {3},
  pages   = {365--380},
  volume  = {1},
  url     = {http://www.mathnet.ru/eng/mmj22},
}

@Article{bernabei2001anomalous,
  author  = {Bernabei, M. S.},
  journal = {Probability Theory and Related Fields},
  title   = {Anomalous behaviour for the random corrections to the cumulants of random walks in fluctuating random media},
  year    = {2001},
  number  = {3},
  pages   = {410--432},
  volume  = {119},
  doi     = {10.1007/PL00008767},
}

@Article{boldrighini2004random,
  author  = {Boldrighini, C. and Minlos, R. A. and Pellegrinotti, A.},
  journal = {Probability Theory and Related Fields},
  title   = {Random walks in quenched i.i.d. space-time random environment are always a.s. diffusive},
  year    = {2004},
  number  = {1},
  pages   = {133--156},
  volume  = {129},
  doi     = {10.1007/s00440-003-0331-x},
}

@Article{rassoulagha2013quenched,
  author        = {Rassoul-Agha, F. and Sepp{\"a}l{\"a}inen, T. and Yilmaz, A.},
  journal       = {Comm. Pure Appl. Math.},
  title         = {Quenched free energy and large deviations for random walks in random potentials},
  year          = {2013},
  number        = {2},
  pages         = {202--244},
  volume        = {66},
  archiveprefix = {arXiv},
  eprint        = {1104.3110},
  publisher     = {Wiley Online Library},
}

@Article{rassoulagha2005almost,
  author   = {Rassoul-Agha, F. and Sepp{\"a}l{\"a}inen, T.},
  journal  = {Probab. Theory Rel. Fields},
  title    = {An almost sure invariance principle for random walks in a space-time random environment},
  year     = {2005},
  number   = {3},
  pages    = {299--314},
  volume   = {133},
  fjournal = {Probability Theory and Related Fields},
}

@Article{rassoulagha2009almost,
  author  = {Rassoul-Agha, F. and Sepp{\"a}l{\"a}inen, T.},
  journal = {Ann. Inst. Henri Poincar{\'e} Probab. Stat},
  title   = {Almost sure functional central limit theorem for ballistic random walk in random environment},
  year    = {2009},
  number  = {2},
  pages   = {373--420},
  volume  = {45},
}

@article{Khasminskii1959,
  author  = {Khas'minskii, R. Z.},
  title   = {On Positive Solutions of the Equation {$\mathfrak{A}u+Vu=0$}},
  journal = {Theory of Probability and its Applications},
  year    = {1959},
  volume  = {4},
  number  = {3},
  pages   = {309--318},
  doi     = {10.1137/1104030}
}

@article{berard2004almost,
author = "Bérard, J.",
fjournal = "Journal of Applied Probability",
journal = "J. Appl. Probab.",
month = "03",
number = "1",
pages = "83--92",
publisher = "Applied Probability Trust",
title = "The almost sure central limit theorem for one-dimensional nearest-neighbour random walks in a space-time random environment",
url = "http://dx.doi.org/10.1239/jap/1077134669",
volume = "41",
year = "2004"
}
\end{document}